\documentclass[12pt]{amsart}

\usepackage[
text={440pt,575pt},
headheight=9pt,
centering
]{geometry}

\usepackage{amsmath, amssymb, amsthm, enumerate, bm}
\allowdisplaybreaks 
\usepackage{mathtools}

\usepackage{graphicx}
\usepackage{subcaption}
\usepackage[normalem]{ulem}

\usepackage[all]{xy}

\usepackage{multirow}

\usepackage{mathptmx}

\usepackage{xcolor}
\usepackage{hyperref}
\definecolor{darkblue}{RGB}{0,0,160}
\hypersetup{
	colorlinks,%
	citecolor=darkblue,%
	filecolor=black,%
	linkcolor=darkblue,%
	urlcolor=darkblue
}

\usepackage[capitalise]{cleveref}

\makeatletter
\newcommand{\nolisttopbreak}{\vspace{\topsep}\nobreak\@afterheading}
\makeatother

\theoremstyle{definition}
\newtheorem{definition}{Definition}[section]

\newtheorem{theorem}[definition]{Theorem}
\newtheorem{proposition}[definition]{Proposition}
\newtheorem{lemma}[definition]{Lemma}
\newtheorem{corollary}[definition]{Corollary}
\newtheorem{claim}{Claim}[definition]

\newtheorem*{fact*}{Fact}
\newtheorem{remark}[definition]{Remark}
\newtheorem{example}[definition]{Example}

\newtheorem{problem}[definition]{Problem}
\newtheorem{conjecture}[definition]{Conjecture}

\newcommand{\C}{\mathcal{C}}

\newcommand{\Hc}{\mathcal{H}}

\renewcommand{\L}{\mathfrak{L}}

\newcommand{\M}{\mathcal{M}}

\newcommand{\NN}{\mathbb{N}}

\newcommand{\OO}{\mathbb{O}}
\newcommand{\RR}{\mathbb{R}}

\newcommand{\ZZ}{\mathbb{Z}}

\newcommand\nub{{\boldsymbol 0}}
\newcommand\ab{{\boldsymbol a}}
\newcommand\bb{{\boldsymbol b}}
\newcommand\cb{{\boldsymbol c}}
\newcommand\db{{\boldsymbol d}}
\newcommand\eb{{\boldsymbol e}}
\newcommand\epb{{\bm \epsilon}}

\newcommand\ib{{\boldsymbol i}}

\newcommand\ub{{\boldsymbol u}}
\newcommand\vb{{\boldsymbol v}}
\newcommand\wb{{\boldsymbol w}}

\DeclareMathOperator{\mn}{mn}

\DeclareMathOperator{\Sym}{Sym}

\DeclareMathOperator{\supp}{supp}

\DeclareMathOperator{\cone}{cone}
\DeclareMathOperator{\conv}{conv}

\DeclareMathOperator{\lin}{lin}

\DeclareMathOperator{\U}{U}

\DeclareMathOperator{\ind}{ind}

\newcommand\defas{\coloneqq}

\makeatletter
\@namedef{subjclassname@2020}{%
	\textup{2020} Mathematics Subject Classification}
\makeatother

\title[Minkowski--Weyl theorem and Gordan's lemma up to symmetry]{
	On Minkowski--Weyl theorem and Gordan's lemma\\ up to symmetry
}

\author{Dinh Van Le}
\address{FPT University, Hanoi, Vietnam}
\email{dinhlv2@fe.edu.vn}


\begin{document}
   
   \begin{abstract} 
   	We investigate equivariant analogues of the Minkowski--Weyl theorem and Gordan's lemma in an infinite-dimensional setting, where cones and monoids are  invariant under the action of the infinite symmetric group. Building upon the framework developed in \cite{KLR,LR21}, we extend the theory beyond the nonnegative case. Our main contributions include a local equivariant Minkowski--Weyl theorem, local-global principles for equivariant finite generation and stabilization of symmetric cones, and a full proof of the equivariant Gordan's lemma conjectured in \cite{LR21}. We also classify non-pointed symmetric cones and non-positive symmetric normal monoids, addressing new challenges in the general setting.
   \end{abstract}
	\maketitle
	
	\section{Introduction}
	
	The  Minkowski--Weyl theorem and Gordan's lemma are foundational results in convex  geometry. The former establishes the duality between two representations of a polyhedral cone: as the conical hull of finitely many vectors, or as the intersection of finitely many linear halfspaces. The latter states that the lattice points within a rational polyhedral cone form an affine monoid. These results are not only cornerstones of polyhedral theory, but also underlie significant developments in areas such as linear and integer programming, invariant theory, toric geometry, algebraic statistics, and the theory of lattice polytopes; see, e.g., \cite{BG,CLS,DHK,Su} for details.
	
	In this paper, we investigate equivariant analogues of the Minkowski--Weyl theorem and Gordan's lemma in an infinite-dimensional setting, where cones and monoids are  invariant under the action of the infinite symmetric group $\Sym$ (see \Cref{section-cones-monoids} for precise definitions). Such symmetry arises naturally in various fields, including algebraic statistics \cite{AHT,AH07,HS12}, commutative algebra \cite{Dr14,DEF,LNNR1,LNNR2,NR17,NR19}, convex optimization \cite{LC}, group theory \cite{Co67}, machine learning \cite{LD,Pc}, and representation theory \cite{CEF,CF,SS14}. 
	
	A foundational framework for the study of $\Sym$-invariant cones and monoids was laid out in \cite{KLR}, with a primary focus on the nonnegative case. There, the authors considered \emph{global} cones $C \subseteq \RR^{(\NN)}_{\ge 0}$ and associated to them {chains} of \emph{local} cones $\C = (C_n)_{n \ge 1}$, where $C_n = C \cap \RR^n$. A key result in that setting was a \emph{local-global principle} which characterizes when a global cone is $\Sym$-equivariantly finitely generated in terms of its local truncations. An analogous principle was established for $\Sym$-invariant monoids in $\ZZ^{(\NN)}_{\ge 0}$.
	
	This line of work was extended in \cite{LR21}, where the authors pursued equivariant extensions of classical results in convex geometry (see \cite[Problem 1.1]{LR21}). In particular, they proposed the following conjectural analogue of Gordan's lemma (see \cite[Conjecture 1.4]{LR21}):
	
	\begin{conjecture}
		\label{cj:equi-Gordan}
		Let $C\subseteq\RR^{(\NN)}$ be a $\Sym$-equivariantly finitely generated rational cone. Then $M=C\cap\ZZ^{(\NN)}$ is a $\Sym$-equivariantly finitely generated normal monoid.
	\end{conjecture}
	
	This conjecture was verified for nonnegative cones in \cite[Theorem 6.1]{LR21}. Moreover, it was also shown that the dual cone of a $\Sym$-equivariantly finitely generated cone need not be $\Sym$-equivariantly finitely generated \cite[Theorem 4.9]{LR21}, indicating that a global equivariant analogue of the Minkowski--Weyl theorem does not hold in general. Nevertheless, a local version of this theorem was established for chains of nonnegative cones \cite[Theorem 4.4]{LR21}, providing an explicit description of the corresponding dual chains.
	
	 The goal of the present paper is to generalize the results of \cite{KLR,LR21} to arbitrary $\Sym$-invariant cones, thus removing the nonnegativity assumption and substantially broadening the scope of the theory. More precisely, we prove:
	 \begin{itemize} 
	 	\item 
	 	a local equivariant analogue of the Minkowski--Weyl theorem for any stabilizing $\Sym$-invariant chain of cones (see \Cref{t:Minkowski-Weyl}); 
	 	\item 
	 	local-global principles characterizing equivariant finite generation and stabilization of general $\Sym$-invariant cones (see \Cref{t:finte-generation,t:stabilizing-cones}); 
	 	\item 
	 	\Cref{cj:equi-Gordan} in its full generality (see \Cref{t:Gordan}).
	 \end{itemize}
	 	
	 Several new challenges arise in this more general setting. For instance, we do not yet know whether a local-global principle holds for all $\Sym$-invariant monoids (see \Cref{pb:local-global-monoids}). However, we are able to establish such a principle under the additional assumption that the monoid is both positive and normal (see \Cref{c:finte-generation-monoid}). Moreover, in contrast to the nonnegative case, a cone $C \subseteq \RR^{(\NN)}$ is not necessarily pointed, and thus the associated monoid $C \cap \ZZ^{(\NN)}$ may fail to be positive. This prevents a direct application of the aforementioned principle for positive normal monoids. To address this, we classify all non-pointed $\Sym$-invariant cones (see \Cref{t:non-pointed-cones}) and all non-positive $\Sym$-invariant normal monoids (see \Cref{c:non-positve-monoids}).
	 
	 The results obtained in this paper contribute to a broader program aimed at understanding algebraic and geometric structures in infinite-dimensional settings up to symmetry. They provide new techniques for analyzing symmetric cones and monoids, and offer a foundation for further developments in this direction.
	 
	 The remainder of the paper is structured as follows. \Cref{section-cones-monoids} provides  preliminaries on symmetric cones and monoids. \Cref{sec-Minkowski-Weyl} develops a duality theory for symmetric chains of cones and presents the proof of a local equivariant Minkowski--Weyl theorem. In \Cref{section-local-global}, we establish local-global principles for arbitrary symmetric cones. Classifications of non-pointed symmetric cones and non-positive symmetric normal monoids are given in \Cref{sec-non-pointed-cones}.
	 In \Cref{sec:equi-Gordan}, we verify \Cref{cj:equi-Gordan} and derive related results. Finally, in \Cref{sec:problems}, we propose two open problems concerning dual monoids and dual lattices, inspired by the duality theory for cones developed earlier.


\section{Symmetric cones and monoids}
\label{section-cones-monoids}

We begin by introducing the infinite-dimensional ambient space $\RR^{(\NN)}$, equipped with the natural action of the infinite symmetric group $\Sym$ via coordinate permutations. Our focus is on cones and monoids in this space that are invariant under the $\Sym$-action. These symmetric structures form the main objects of study in this paper.

Let $\NN$ denote the set of positive integers. Throughout the paper, we consider the infinite dimensional real vector space
\[
\RR^{(\NN)}=\bigoplus_\NN\RR,
\]
which has a basis indexed by $\NN$. Any element $\ub\in\RR^{(\NN)}$ can be expressed as $\ub=(u_i)_{i\in \NN}$, where each coordinate $u_i\in\RR$ and only finitely many of them are nonzero. We call
\[
\supp(\ub)=\{i\in\NN\colon u_i\ne 0\}
\]
the \emph{support} of $\ub$. The number $|\supp(\ub)|$ is referred to as the \emph{support size} of $\ub.$ 
Let $\{\eb_i\}_{i\in\NN}$ denote the standard basis of $\RR^{(\NN)}$, where $\eb_i\in\RR^{(\NN)}$ is the vector consisting of all zeroes except in the $i$-th coordinate, which is 1. 
For any $n\in\NN$, we regard $\RR^n$ as the subspace of $\RR^{(\NN)}$ spanned by $\eb_1,\dots,\eb_n$. That is, each element $(u_1,\dots,u_n)\in\RR^n$ is identified with the vector
$(u_1,\dots,u_n,0,0,\dots)\in\RR^{(\NN)}$, where all coordinates beyond the $n$-th position are zero.
Via this identification, $\RR^{(\NN)}$ is realized as the direct limit of the following ascending chain of finite dimensional subspaces:
\[
\RR\subset \RR^2\subset\cdots\subset\RR^n\subset\cdots.
\]

\subsection{Symmetric groups and their actions}
\label{sec:Sym}
For $n\in \NN$, let $\Sym(n)$ denote the symmetric group on $[n]\defas\{1,\dots,n\}$.  By identifying $\Sym(n)$ with the stabilizer subgroup of $n+1$ in $\Sym(n+1)$, we obtain an ascending chain of finite symmetric groups:
$$
\Sym(1)\subset \Sym(2)\subset\cdots\subset \Sym(n)\subset\cdots.
$$
The direct limit of this chain is the infinite symmetric group
\[
\Sym(\infty)=\bigcup_{n\geq 1} \Sym(n),
\]
which consists of all permutations of $\NN$ that fix all but finitely many elements. For brevity, we will henceforth refer to $\Sym(\infty)$ simply as $\Sym$.

The group $\Sym$ acts on $\RR^{(\NN)}$ by permuting coordinates. More specifically, for $\sigma\in\Sym$ and $\ub=(u_i)_{i\in\NN}=\sum_{i\in\NN}u_i\eb_i\in\RR^{(\NN)}$, the action is given by
\[
\sigma(\ub)
=(u_{\sigma^{-1}(1)},u_{\sigma^{-1}(2)},\dots)
=\sum_{i\in\NN}u_{\sigma^{-1}(i)}\eb_i
=\sum_{i\in\NN}u_i\eb_{\sigma(i)}.
\]
In particular, $\sigma(\eb_i)=\eb_{\sigma(i)}$ for all $i\in\NN$. Evidently, this action induces an action of $\Sym(n)$ on $\RR^{n}$ for each $n\ge1$. 

For a subset $A\subseteq\RR^{(\NN)}$ and a subgroup $\Gamma\subseteq\Sym$ denote 
\[
\Gamma(A)=\{\sigma(\ub)\mid \sigma\in \Gamma,\ \ub\in A\}.
\]
We say that $A$ is \emph{$\Gamma$-invariant} if $\Gamma(A)\subseteq A$.
In this paper, we mainly focus on the case $\Gamma=\Sym$ or $\Gamma=\Sym(n)$ for some $n\ge1$. An ascending chain
\[
A_1\subset A_2\subset\cdots\subset A_n\subset\cdots,
\]
where $A_n\subseteq\RR^n$ for all $n\ge1$, is called \emph{$\Sym$-invariant} if each $A_n$ is $\Sym(n)$-invariant. For such a chain, one checks readily that its direct limit $A=\bigcup_{n\geq 1} A_n$ is $\Sym$-invariant. Conversely, given a $\Sym$-invariant subset $A\subseteq\RR^{(\NN)}$, the ascending chain of its truncations:
\[
A\cap\RR\subset A\cap\RR^2\subset\cdots\subset A\cap\RR^n\subset\cdots
\]
is $\Sym$-invariant. In what follows, we typically use the term \emph{global}  to refer to a set, cone, or monoid 
$A\subseteq\RR^{(\NN)}$, while a set, cone, or monoid $A_n\subseteq\RR^n$ is referred to as \emph{local}.

\subsection{Symmetric cones}
Given a subset $A\subseteq\RR^{(\NN)}$, the \emph{cone generated by $A$} is defined as
\[
\cone(A) =
\RR_{\geq 0}A
\defas
\Big\{\sum_{i=1}^k\lambda_i\ab_i\mid\ab_i\in A,\
\lambda_i\in\RR_{\geq 0}, \  k\in\NN\Big\}.
\]
 This cone is said to be \emph{finitely generated} if $A$ is finite, and \emph{rational} if $A\subseteq\ZZ^{(\NN)}$. Observe that any finitely generated cone $C\subseteq\RR^{(\NN)}$ must be contained in $\RR^{n}$ for some $n\in\NN$.  The classical Minkowski--Weyl theorem states that every finitely generated cone $C\subseteq\RR^{n}$ can equivalently be represented as the intersection of finitely many linear halfspaces of $\RR^{n}$ (see, for example, \cite[Theorem 1.15]{BG}). The proof of this result, which utilizes Fourier--Motzkin elimination, further shows that when $C$ is rational, it is the intersection of linear \emph{rational} halfspaces, i.e., halfspaces defined by linear forms with integer coefficients.

Now suppose that $C\subseteq\RR^{(\NN)}$ is a $\Gamma$-invariant cone, where $\Gamma=\Sym$ or $\Gamma=\Sym(n)$ for some $n\ge1$. Then $C$  always has a \emph{$\Gamma$-equivariant generating set},  meaning a subset $A\subseteq\RR^{(\NN)}$ such that $\Gamma(A)$ generates $C$:
\[
C=\cone(\Gamma(A)) =
\Big\{\sum_{i=1}^k\lambda_i\sigma_i(\ab_i)\mid\ab_i\in A,\ \sigma_i\in \Gamma,\
\lambda_i\in\RR_{\geq 0}, \  k\in\NN\Big\}.
\]
For example, the cone $C$ itself  forms a (possibly infinite) $\Gamma$-equivariant generating set for $C$.
We say that $C$ is \emph{$\Gamma$-equivariantly finitely generated} if it admits a finite $\Gamma$-equivariant generating set. As we will see in \Cref{section-local-global}, the $\Sym$-equivariant finite generation of a $\Sym$-invariant global cone $C\subseteq\RR^{(\NN)}$ can be characterized in terms of any $\Sym$-invariant chain of local cones whose direct limit is $C$. To prepare for this, we introduce here the relevant terminology. Let $\C=(C_{n})_{n\geq 1}$ be a \emph{$\Sym$-invariant chain of cones}. This means that $\C$ is an ascending chain and each $C_n$ is a $\Sym(n)$-invariant cone in $\RR^n$. Consequently, we have the containment
\[
\cone(\Sym(n)(C_m))\subseteq C_n
\quad\text{for all }\ n\ge m\ge 1.
\]
The chain $\C$ is said to \emph{stabilize} if there exists $r\in\NN$ such that
\[
C_n=\cone(\Sym(n)(C_m))
\quad\text{for all }\ n\ge m\ge r.
\]
The \emph{stability index} of $\C$, denoted $\ind(\C)$, is the smallest such $r$.

\subsection{Symmetric monoids}
Monoids are discrete analogues of cones. A monoid $M\subseteq\ZZ^{(\NN)}$ is \emph{generated} by a set $A$, denoted $M=\mn(A)$, if
\[
M =
\ZZ_{\geq 0}A
\defas
\Big\{\sum_{i=1}^km_i\ab_i\mid\ab_i\in A,\
m_i\in\ZZ_{\geq 0}, \  k\in\NN\Big\}.
\]
This monoid is called \emph{affine} or \emph{finitely generated} if it has a finite generating set. 
The \emph{subgroup of units} of $M$ is given by
\[
\U(M)=\{\ub\in M\mid -\ub\in M\}.
\] 
We say that $M$ is \emph{positive} if $\U(M)=\{0\}$. This condition holds, for instance, when $M\subseteq\ZZ^{\NN}_{\ge0}$. If $M$ is a positive affine monoid, then it has a unique minimal generating set known as the \emph{Hilbert basis} of $M$, denoted $\Hc(M)$ (see, e.g., \cite[Definition 2.15]{BG}). Moreover, in this case, $\Hc(M)$ consists precisely of the \emph{irreducible} elements of $M$, i.e., those elements $\ub\in M\setminus\{\nub\}$ with the property that if $\ub=\vb+\wb$ for $\vb,\wb\in M$, then either $\vb=\nub$ or $\wb=\nub$.

The \emph{normalization} of $M$ is defined as
\[
\widehat{M}=\{\ub\in\ZZ^{(\NN)}\colon k\ub\in M\ \text{ for some }\ k\in\NN\}.
\]
We say that $M$ is \emph{normal} if $M=\widehat{M}.$ Normal monoids are closely related to cones. In fact, it is straightforward to see that $\widehat{M}=\cone(M)\cap\ZZ^{(\NN)}$, where $\cone(M)$ is the cone generated by $M$. Thus, if $M$ is normal, then it can be expressed as $M=\cone(M)\cap\ZZ^{(\NN)}$. Conversely, for any cone $C\subseteq\RR^{(\NN)}$, the intersection $C\cap\ZZ^{(\NN)}$ defines a normal monoid. In this setting, the classical Gordan's lemma asserts that if $C$ is a finitely generated rational cone, then the normal monoid $C\cap\ZZ^{(\NN)}$ is affine (see, e.g., \cite[Lemma 2.9]{BG}).

As with cones, one can define \emph{equivariant generating sets} for $\Sym(n)$- and $\Sym$-invariant monoids in $\ZZ^{(\NN)}$, along with the notion of \emph{stabilization} for $\Sym$-invariant chains of local monoids. Similarly, \emph{equivariant Hilbert bases} can be defined for $\Sym(n)$- and $\Sym$-invariant positive monoids. We omit the analogous details here. In \Cref{section-local-global}, we will provide a characterization of equivariant finite generation of $\Sym$-invariant normal monoids in terms of their associated $\Sym$-invariant chains of local monoids.

\section{Equivariant Minkowski--Weyl theorem}
\label{sec-Minkowski-Weyl}

An alternative formulation of the Minkowski--Weyl theorem states that the dual cone of any finitely generated cone in $\RR^{n}$ is also finitely generated. This section aims to extend this classical result to the equivariant setting. While a \emph{global} extension to arbitrary $\Sym$-invariant cones $C\subseteq\RR^{(\NN)}$ does not hold, we establish a \emph{local} extension that describes the dual cones of any stabilizing $\Sym$-invariant chain of local cones.

\subsection{Dual cones}
Let $\RR^{\NN}$ denote the Cartersian power of $\RR$ indexed by $\NN$. This vector space contains $\RR^{(\NN)}$ as a subspace, and it elements take the form $(v_i)_{i\in\NN}$, where $v_i\in\RR$ for all $i\in\NN$. We identify $\RR^{\NN}$ with the dual vector space of $\RR^{(\NN)}$ via the dual pairing
\begin{align*}
	\langle\cdot,\cdot\rangle:\RR^{(\NN)}\times\RR^\NN&\to\RR
\end{align*}
given by
\[
\langle\ub,\vb\rangle=\sum_{i\in \NN}u_iv_i
\ \text{ for any $\ub=(u_i)_{i\in\NN}\in\RR^{(\NN)}$ and $\vb=(v_i)_{i\in\NN}\in\RR^{\NN}$. }
\]
When restricted to $\RR^n\times\RR^n$, this pairing yields the standard identification of the dual space of $\RR^n$ with itself for every $n\in\NN$.

Given a cone $C\subseteq\RR^{(\NN)}$, its \emph{dual cone} is defined as
	\[
C^*=\{\vb\in \RR^{\NN}\mid \langle \ub,\vb\rangle\ge 0\
\text{ for all }\ \ub\in C\}.
\]
Similarly, for a fixed $n\in\NN$, the \emph{dual cone} of a cone $C_n\subseteq\RR^n$  is defined by
\[
C_n^*=\{\vb\in \RR^{n}\mid \langle \ub,\vb\rangle\ge 0\
\text{ for all }\ \ub\in C_n\}.
\]
Note that if $C_n\subseteq\RR^{n}$, then it is also a subset of $\RR^{m}$ for any $m\ge n$. Consequently, the definition of the dual cone of $C_n$ \emph{depends} on the ambient space in which it is considered. Nevertheless, this dependence causes no confusion in our later discussion because, for a chain of cones  $\C=(C_{n})_{n\geq 1}$, we always consider the dual cone of $C_n$ with respect to the embedding $C_n\subseteq\RR^n$.

If a cone $C_n\subseteq\RR^n$ is the intersection of finitely many halfspaces, then its dual cone $C_n^*$ is generated by the linear forms defining these halfspaces (see, e.g., \cite[Theorem 1.16]{BG}). This duality leads to the alternative formulation of the Minkowski--Weyl theorem mentioned at the begining of this section. We now turn to equivariant extensions of this result. Note that the dual of a $\Sym$-invariant (or $\Sym(n)$-invariant) cone is again $\Sym$-invariant (respectively, $\Sym(n)$-invariant); see \cite[Lemma 4.3]{LR21}.

\subsection{Equivariant Minkowski--Weyl theorem: global extension}

 Given a $\Sym$-invariant cone $C\subseteq\RR^{(\NN)}$, it is shown in \cite[Theorem 4.9]{LR21} that its dual cone $C^*$ need not be $\Sym$-equivariantly finitely generated, even when $C$ itself is. This demonstrates that the classical Minkowski--Weyl theorem does not have a direct global extension to the equivariant setting. However, when restricting the dual cone to the subspace $\RR^{(\NN)}$ of $\RR^{\NN}$, we obtain the following proposition.

\begin{proposition}
	\label{p:global-W-M}
Let $C\subseteq\RR^{(\NN)}$ be any $\Sym$-invariant cone. Then the restricted dual cone $C^*\cap \RR^{(\NN)}$ is one of the following:
$$\{\nub\}, \quad \RR_{\ge0}^{(\NN)}, \quad \RR_{\le0}^{(\NN)},\quad \RR^{(\NN)}.$$
 In particular, $C^*\cap \RR^{(\NN)}$ is always $\Sym$-equivariantly finitely generated.
\end{proposition}

The proof of this result relies on the following lemma.

\begin{lemma}
	\label{l:dual-orbit}
	Let $\ub,\vb\in \RR^{(\NN)}\setminus\{\nub\}$ be such that
	\[
	\langle\sigma(\ub),\tau(\vb)\rangle\ge0
	\quad\text{for all } \sigma, \tau\in \Sym.
	\]
	Then both $\ub$ and $\vb$ must belong to either $\RR_{\ge0}^{(\NN)}$ or  $\RR_{\le0}^{(\NN)}$. 
\end{lemma}

\begin{proof}
	First, we show that $\ub$ belongs to either $\RR_{\ge0}^{(\NN)}$ or  $\RR_{\le0}^{(\NN)}$. Assume, for contradiction, that this is not the case. Then there exists $\sigma\in \Sym$ such that
	\[
	\sigma(\ub)=(a_1,a_2,\dots,a_n,0,0,\dots)
	\]
	for some $n\ge 2$, where $a_1>0$, $a_2<0$, and $a_i=0$ for all $i\ge n+1$. Since $\vb\ne\nub$, we can find $\tau_1,\tau_2\in \Sym$ such that
	\[
	\tau_1(\vb)=(b_1,0,\dots,0,b_{n+1},b_{n+2},\dots),\quad 
	\tau_2(\vb)=(0,b_1,0,\dots,0,b_{n+1},b_{n+2},\dots),
	\]
	where $b_1\ne0$ and $b_i=0$ for $i=3,\dots,n$. We then have
	\[
	\langle\sigma(\ub),\tau_1(\vb)\rangle=a_1b_1\ \text{ and }\
	\langle\sigma(\ub),\tau_2(\vb)\rangle=a_2b_1.
	\]
	Since $b_1 \ne 0$ and $a_1,a_2$ have opposite signs, one of these inner products must be negative, contradicting the assumption that all such inner products are nonnegative. This contradiction confirms that $\ub\in \RR_{\ge0}^{(\NN)}\cup\RR_{\le0}^{(\NN)}$. Applying the same argument to $\vb$ yields $\vb\in \RR_{\ge0}^{(\NN)}\cup\RR_{\le0}^{(\NN)}$. Finally, since $\langle \ub,\vb\rangle\ge0$, both $\ub$ and $\vb$ must belong to the same orthant.
\end{proof}

We are ready to prove \Cref{p:global-W-M}.

\begin{proof}[Proof of \Cref{p:global-W-M}]
It suffices to prove the first assertion since the cones $\RR_{\ge0}^{(\NN)},$ $\RR_{\le0}^{(\NN)}$, and $\RR^{(\NN)}$ are $\Sym$-equivariantly generated by $\eb_1$, $-\eb_1$, and $\pm \eb_1$, respectively. If $C=\{\nub\}$, then $C^*=\RR^{\NN}$, and hence $C^*\cap \RR^{(\NN)}=\RR^{(\NN)}$. Now assume that $C\ne\{\nub\}$. We consider the following cases:

\emph{Case 1}: $C\subseteq\RR_{\ge0}^{(\NN)}$. Then $\RR_{\ge0}^{\NN}\subseteq C^*$, so $\RR_{\ge0}^{(\NN)}\subseteq C^*\cap\RR^{(\NN)}$. On the other hand, it follows easily from \Cref{l:dual-orbit} that $C^*\cap\RR^{(\NN)}\subseteq\RR_{\ge0}^{(\NN)}$. Therefore, $C^*\cap\RR^{(\NN)}=\RR_{\ge0}^{(\NN)}$.

\emph{Case 2}: $C\subseteq\RR_{\le0}^{(\NN)}$. Arguing similarly to Case 1, we obtain $C^*\cap\RR^{(\NN)}=\RR_{\le0}^{(\NN)}$.

\emph{Case 3}: $C\nsubseteq\RR_{\ge0}^{(\NN)}\cup\RR_{\le0}^{(\NN)}$. In this case, there exists $\ub \in C$ with both positive and negative coordinates. Take any $\vb \in C^*\cap\RR^{(\NN)}$. Since $\langle \sigma(\ub), \tau(\vb) \rangle \ge 0$ for all $\sigma, \tau \in \Sym$, \Cref{l:dual-orbit} implies that $\vb=\nub$. Hence, $C^*\cap\RR^{(\NN)}=\{\nub\}$. This completes the proof.
\end{proof}

\subsection{Equivariant Minkowski--Weyl theorem: local version}

In view of \Cref{p:global-W-M}, although the restricted dual cone $C^*\cap \RR^{(\NN)}$ is $\Sym$-equivariantly finitely generated, it offers limited insight into the full dual cone $C^*$ (as well as the original cone $C$). To gain a better understanding of $C^*$, we consider the chain $\C^*=(C^*_{n})_{n\geq 1}$ of dual cones associated with a chain $\C=(C_{n})_{n\geq 1}$ of local cones derived from $C$. As will be shown in \Cref{t:finte-generation}, if the global cone $C$ is $\Sym$-equivariantly finitely generated, then the chain $\C$ stabilizes and eventually becomes  finitely generated. In this case, we are able to explicitly describe a generating set for $C^*_{n}$ when $n$ is large enough, thereby providing further information about the global dual cone $C^*$. This result may be viewed as a local equivariant extension of the Minkowski--Weyl theorem. To state it precisely, we first introduce some notation.

For $n\ge 1$, denote by $\OO(\RR^n)$ and $\OO^-(\RR^n)$ the subsets of $\RR^n$ consisting of non-decreasing and non-increasing vectors, respectively. That is,
\begin{align*}
\OO(\RR^n)&=\{(u_1,\dots,u_{n})\in \RR^{n}\mid u_1\le\dots\le u_{n}\},\\
\OO^-(\RR^n)&=\{(v_1,\dots,v_{n})\in \RR^{n}\mid v_1\ge\dots\ge v_{n}\}.
\end{align*}
More generally, for any subset $D\subseteq \RR^n$, we define $\OO(D)\defas\OO(\RR^n)\cap D.$ In particular, if $a,b\in\RR$ with $a\le b$, then $\OO([a,b]^n)$ represents the set of non-decreasing vectors in $\RR^n$ where each entry lies between $a$ and $b$ (inclusive). Let $I_n(a,b)$ be the subset  of $\OO([a,b]^n)$ consisting of vectors whose entries are either $a$ or $b$. For instance,
\begin{align*}
I_1(a,b)&=\{a, b\},\\
I_2(a,b)&=\{(a,a), (a,b), (b,b)\},\\
I_3(a,b)&=\{(a,a,a), (a,a,b), (a,b,b), (b,b,b)\}.
\end{align*}
It is important to note that if $a<b$, then the set $I_n(a,b)$ consists of exactly $n+1$ affinely independent elements.

Now, let $\ub=(u_1,\dots,u_{n})\in \OO(\RR^{n})$ and fix an index $i\in[n-1]$. For $m>n$, define $F_{i,m}(\ub)$ as the subset of $\OO(\RR^m)$ consisting of vectors obtained from $\ub$ by inserting elements of $I_{m - n}(u_i, u_{i + 1})$ between $u_i$ and $u_{i + 1}$:
\[
F_{i,m}(\ub)=\{(u_1,\dots,u_{i})\}\times I_{m-n}(u_i,u_{i+1})
    \times\{(u_{i+1},\dots,u_n)\}.
\] 
For example, if $\ub=(1,2,3,4)\in \OO(\RR^{4})$ and $i=2$, then
\begin{align*}
F_{2,5}(\ub)&=\{(1,2,2,3,4), (1,2,3,3,4)\},\\
F_{2,6}(\ub)&=\{(1,2,2,2,3,4), (1,2,2,3,3,4),(1,2,3,3,3,4)\},\\
F_{2,7}(\ub)&=\{(1,2,2,2,2,3,4), (1,2,2,2,3,3,4),(1,2,2,3,3,3,4),(1,2,3,3,3,3,4)\}.
\end{align*}

We are prepared to state a local equivariant extension of the Minkowski--Weyl theorem. The following result extends the case of chains of nonnegative cones established in \cite[Theorem 4.4]{LR21}. Recall that for a chain of cones  $\C=(C_{n})_{n\geq 1}$, the dual cone $C_n^*$ is always considered with respect to the embedding $C_n\subseteq\RR^n$.

\begin{theorem}
	\label{t:Minkowski-Weyl}
	Let $\C=(C_{n})_{n\geq 1}$ be a stabilizing $\Sym$-invariant chain of cones with stability index $r=\ind(\C).$ Denote by $\C^*=(C_{n}^*)_{n\geq 1}$ the corresponding chain of dual cones. Suppose $F_{2r}\subseteq \OO(\RR^{2r})$ is a generating set for $\OO(C_{2r}^*)$. Then for all $n>2r$, the set
    $$F_n\defas\bigcup_{\ub\in F_{2r}} F_{r,n}(\ub)$$
	is a $\Sym(n)$-equivariant generating set for $C_n^*$.
\end{theorem}

We illustrate this result with the following example.

\begin{example}
	\label{e:dual-cone}
	Consider a $\Sym$-invariant chain $\C=(C_{n})_{n\geq 1}$ with stability index $\ind(\C)=3$, where the cone $C_3$ is $\Sym(3)$-equivariantly generated by the two vectors:
	\[
	(-2,-1,4)\quad\text{and}\quad (-3,1,3).
	\]
	Using Normaliz \cite{BISV} and Macaulay2 \cite{GS}, we find the following generating set for $\OO(C_{6}^*)$:
	\[
	F_6=\{(1^{(6)}),\ (5,6^{(4)},7), \  (3^{(5)},4),\ (3^{(4)},4^{(2)}),\ (3^{(3)},4^{(3)}),\ (3^{(2)},4^{(4)}),\ (3,4^{(5)})\},
	\]
	where the notation $a^{(k)}$ indicates that the entry $a$ is repeated $k$ times. By \Cref{t:Minkowski-Weyl}, for all $n \geq 7$, the cone $C_{n}^*$ is $\Sym(n)$-equivariantly generated by the set
	\[
	F_n=\bigcup_{\ub\in F_{6}} F_{3,n}(\ub)
	=\{(1^{(n)}),\ (5,6^{(n-2)},7)\}\cup \{(3^{(i)},4^{(n-i)})\mid i\in[n-1]\}.
	\]
	It is not hard to show that $(1^{(n)})$ is a redundant generator of $C_{n}^*$ (see \Cref{c:entire-cone}). Therefore, for all $n \geq 7$, $C_{n}^*$ admits the following $\Sym(n)$-equivariant generating set:
	\[
	F_n'
	=\{(5,6^{(n-2)},7)\}\cup \{(3^{(i)},4^{(n-i)})\mid i\in[n-1]\}.
	\]
\end{example}

\begin{remark}
	\label{r:M-W}
	Keep the notation of \Cref{t:Minkowski-Weyl}.
	\begin{enumerate}
		\item 
		For any $\Sym(n)$-invariant cone $D_{n}\subseteq\RR^n$, we have $D_{n}=\Sym(n)(\OO(D_{n}))$. This implies that any generating set for $\OO(D_{n})$ is a $\Sym(n)$-equivariant generating set for $D_n$. Thus, in particular, $F_{2r}$ is a $\Sym(2r)$-equivariant generating set for $C_{2r}^*$. On the other hand, a $\Sym(n)$-equivariant generating set for $D_{n}$ contained in $\OO(D_{n})$ need not generate $\OO(D_{n})$. For instance, consider the cone $D_3=\cone(\Sym(3)(\ub))$ with $\ub=(1,2,3)$. Then $\{\ub\}$ is a $\Sym(3)$-equivariant generating set for $D_{3}$, but it is not a generating set for $\OO(D_{3})$, since $$(1,1,1)=\frac{1}{4}\big((1,2,3)+(3,2,1)\big)\in\OO(D_{3})\setminus\cone(\ub).$$
		\item 
		If the chain $\C$ consists of finitely generated cones, then by the Minkowski--Weyl theorem, the dual chain $\C^*$ also consists of finitely generated cones. In particular, it follows from \cite[Lemma 4.6]{LR21} that the cone $\OO(C_{2r}^*)$ is finitely generated. Thus, $F_{2r}$ can be chosen to be finite, and \Cref{t:Minkowski-Weyl} provides a finite $\Sym(n)$-equivariant generating set for $C_n^*$ for all $n>2r$. 
		\item 
		\Cref{t:Minkowski-Weyl} does not fully generalize \cite[Theorem 4.4]{LR21}. In the case of nonnegative cones, the latter yields a $\Sym(n)$-equivariant generating set for $C_n^*$ for all $n>r$, whereas the former only guarantees such a set for $n>2r$. While a full generalization of \cite[Theorem 4.4]{LR21} is indeed possible (see  \Cref{r:W-M-improved} below), we have opted for the current version of \Cref{t:Minkowski-Weyl} for the sake of simplicity and conciseness.
	\end{enumerate}
\end{remark}

\enlargethispage{1.2\baselineskip}

The proof of \Cref{t:Minkowski-Weyl} requires some preparations. Recall that the \emph{convex hull} of a subset $A\subseteq\RR^{(\NN)}$ is given by
\[
\conv(A)=
\Big\{\sum_{i=1}^k\lambda_i\ab_i\mid\ab_i\in A,\
\lambda_i\geq 0, \  \sum_{i=1}^k\lambda_i=1,\ k\in\NN\Big\}.
\]
Clearly, $\conv(A)\subseteq\cone(A).$

\begin{lemma}
\label{l:simplex}
    Let $a,b\in\RR$ with $a\le b$. Then $\OO([a,b]^n)=\conv(I_n(a,b))$.
\end{lemma}

\begin{proof}
    Let $P=\conv(I_n(a,b))$. It suffices to show that $\OO([a,b]^n)\subseteq P$, as the reverse inclusion is straightforward. The statement is trivially true when $a=b$, since both sets reduce to the singleton $\{(a,\dots,a)\}$. Now, suppose that $a<b$. Then $I_n(a,b)$ is an affinely independent set with $n+1$ elements, so $P$ is an $n$-dimensional simplex. It follows that $P$ is the intersection of $n+1$ half-spaces, each bounded by a supporting hyperplane that passes through exactly $n$ elements of $I_n(a,b)$. In other words, $P$ is defined by the following inequalities:
    \[
    a\le x_1,\ x_1\le x_2,\dots, x_{n-1}\le x_n,\ x_n\le b.
    \]
    Evidently, any element of $\OO([a,b]^n)$ satisfies all these inequalities. Therefore, $\OO([a,b]^n)\subseteq P$, as required.
\end{proof}

Let $\ub\in\OO(\RR^n)$. For $i\in [n-1]$  and $m>n$, denote by $\OO_{i,m}(\ub)$ the subset of $\OO(\RR^m)$ obtained from $\ub$ by inserting elements of $\OO([u_i,u_{i+1}]^{m-n})$ between $u_i$ and $u_{i+1}$, that is,
\begin{align*}
	\OO_{i,m}(\ub)
	&= \{(u_1,\dots,u_{i})\}\times \OO([u_i,u_{i+1}]^{m-n})
	\times\{(u_{i+1},\dots,u_n)\}.
\end{align*}

\begin{lemma}
	\label{l:convex-hull}
	Let $\ub\in\OO(\RR^n)$. Then for any $i\in [n-1]$  and $m>n$, we have
	$$
	\OO_{i,m}(\ub)
	=\conv(F_{i,m}(\ub)).
	$$
\end{lemma}

\begin{proof}
	By \Cref{l:simplex}, we know that $\OO([u_i,u_{i+1}]^{m-n})=\conv(I_{m-n}(u_i,u_{i+1}))$. Hence,
	$$
		\OO_{i,m}(\ub)
		=\{(u_1,\dots,u_{i})\}\times \conv(I_{m-n}(u_i,u_{i+1}))
		\times\{(u_{i+1},\dots,u_n)\}
		=\conv(F_{i,m}(\ub)),
	$$
	as claimed.
\end{proof}

Recall that the \emph{Minkowski sum} of subsets $A_1,\dots,A_s\subseteq\RR^{(\NN)}$ with weights $\lambda_1,\dots,\lambda_s\in\RR$ is defined as
\[
\sum_{j=1}^s\lambda_jA_j=\Big\{\sum_{j=1}^s\lambda_j\ab_j\mid \ab_j\in A_j \text{ for } j=1,\dots,s\Big\}
\subseteq\RR^{(\NN)}.
\]
Obviously, if $\lambda_j\ge 0$ for all $j\in[s]$, then $\sum_{j=1}^s\lambda_jA_j\subseteq\cone\big(\bigcup_{j=1}^sA_j\big)$.

\begin{lemma}
	\label{l:Fim}
	Suppose $\ub\in\cone(A)$ for some $A\subseteq \RR^n$. Then for any $i\in [n-1]$  and $m>n$, it holds that
	$$
	F_{i,m}(\ub)\subseteq
	\cone\Big(\bigcup_{\ab\in A}F_{i,m}(\ab)\Big).
	$$
\end{lemma}

\begin{proof}
We can write $\ub=\sum_{j=1}^s\lambda_j\ab_j$ with $\ab_j\in A$ and $\lambda_j\ge0$ for all $j\in[s]$. From the definition of $F_{i,m}(\ub)$, it is easy to check that
$
F_{i,m}(\ub)\subseteq
\sum_{j=1}^s\lambda_j F_{i,m}(\ab_j).
$
Hence, 
\[F_{i,m}(\ub)\subseteq
\cone\Big(\bigcup_{j=1}^sF_{i,m}(\ab_j)\Big)\subseteq
\cone\Big(\bigcup_{\ab\in A}F_{i,m}(\ab)\Big).
\hfil
\qedhere
\]
\end{proof}

For the proof of  the next lemma, we will employ the classical \emph{rearrangement inequality} \cite[Theorem 368]{HLP}, which states that for any two non-decreasing sequences of real numbers $a_1\le\cdots\le a_n$, $b_1\le\cdots\le b_n$, and for any permutation $\sigma\in\Sym(n)$, the following inequality holds:
\[
a_{\sigma(1)}b_1+\cdots+ a_{\sigma(n)}b_n
\ge
a_nb_1+\cdots+ a_1b_n.
\]
In other words, if $\ab\in\OO^-(\RR^n)$ and $\bb\in\OO(\RR^n)$, then for any $\sigma\in\Sym(n)$, we have
\[
\langle \sigma(\ab),\bb\rangle
\ge
\langle \ab,\bb\rangle.
\]

\begin{lemma}
\label{l:padding}
    Let $\ab\in\OO^-(\RR^n)$ and $\bb\in\OO(\RR^n)$ with $\langle \ab,\bb\rangle\ge 0$. Suppose there exists an index $i\in [n-1]$ such that $a_i\ge0\ge a_{i+1}$. Then for any $m>n$, we have
        \[
        \langle \sigma(\ab),\db\rangle\ge 0
        \ \text{  for all $\db\in \OO_{i,m}(\bb)$ and $\sigma\in\Sym(m)$.}
        \]
\end{lemma}

\begin{proof}
     Let $\cb\in\OO^-(\RR^m)$ be the vector obtained from $\ab$ by inserting $m-n$ zeros between $a_i$ and $a_{i+1}$, i.e.,
    \[
    \cb=(a_1,\dots,a_i,0,\dots,0,a_{i+1},\dots,a_n).
    \]
    Then $\cb=\tau(\ab)$ for some $\tau\in\Sym(m)$. Moreover, for any $\db\in \OO_{i,m}(\bb)$, it is easily seen that $\langle \cb,\db\rangle=\langle \ab,\bb\rangle$. Since $\cb\in\OO^-(\RR^m)$ and $\db\in\OO(\RR^m)$, the rearrangement inequality yields
    \[
    \langle \sigma(\ab),\db\rangle
    =\langle \sigma\circ\tau^{-1}(\cb),\db\rangle
    \ge \langle \cb,\db\rangle
    =\langle \ab,\bb\rangle\ge 0
    \]
    for all $\sigma\in\Sym(m)$.
\end{proof}

We also need the following basic yet crucial relationship between successive dual cones.

\begin{lemma}
	\label{l:dual-cone}
	Let $\C=(C_{n})_{n\geq 1}$ be a $\Sym$-invariant chain of cones, and let $\C^*=(C_{n}^*)_{n\geq 1}$ denote the corresponding chain of dual cones. If $\vb=(v_1,\dots,v_{n+1})\in C_{n+1}^*$, then for any $i\in[n+1]$, the vector $\hat\vb=(v_1,\dots,v_{i-1},v_{i+1},\dots,v_{n+1})$ belongs to $C_{n}^*.$
\end{lemma}

\begin{proof}
	For any $\ub=(u_1,\dots,u_{n})\in C_{n}$, define $\tilde{\ub}=(u_1,\dots,u_{i-1},0,u_{i},\dots,u_{n}).$
	Then it is clear that $\tilde{\ub}\in C_{n+1}$, and that $\langle \ub,\hat{\vb}\rangle=\langle \tilde\ub,\vb\rangle\ge 0$, where the inequality holds because $\vb\in C_{n+1}^*$.  Hence, $\hat\vb\in C_{n}^*$, as claimed.
\end{proof}

With these preparations in place, we are now ready to prove \Cref{t:Minkowski-Weyl}.

\begin{proof}[Proof of \Cref{t:Minkowski-Weyl}]
	We first show that $F_n\subseteq C_n^*$ for all $n>2r.$ Let $G_r\subseteq\OO^-(C_r)$ be a $\Sym(r)$-equivariant generating set for $C_r$. For instance, one may simply take $G_r=\OO^-(C_r)$, since $C_r=\Sym(r)(\OO^-(C_r))$. Given that $\ind(\C)=r$, it is easy to see that $G_r$ is  a $\Sym(n)$-equivariant generating set for $C_n$ for all $n\ge r$. Thus, in order to prove that $F_n\subseteq C_n^*$ for all $n>2r$, it suffices to verify that
	\[
	\langle \sigma(\ub),\vb\rangle\ge 0
	\ \text{ for all }\ \ub\in G_r, \ \vb\in F_n, \text{ and } \sigma\in\Sym(n).
	\]
	Fix $\ub=(u_1,\dots,u_r)\in G_r$, $\vb\in F_n$, and $\sigma\in\Sym(n)$. Adopting the convention that $u_0 = +\infty$ and $u_{r+1} = -\infty$, we can find an index $j\in \{0,1,\dots,r\}$ such that $u_j\ge0\ge u_{j+1}$. Define $\tilde{\ub}\in\OO^-(\RR^{2r})$ to be the vector obtained from $\ub$ by inserting $r$ zeros between $u_j$ and $u_{j+1}$:
	\[
	\tilde{\ub}=(u_1,\dots,u_{j},0,\dots,0,u_{j+1},\dots,u_{r})\in \OO^-(\RR^{2r}).
	\]
	 Then $\tilde{\ub}=\tau(\ub)$ for some $\tau\in\Sym(2r)$, and so $\tilde{\ub}\in C_{2r}$. Moreover, by construction, the $r$-th and $(r+1)$-th entries of $\tilde{\ub}$ satisfy $\tilde{u}_r\ge 0\ge \tilde{u}_{r+1}$. Since $\vb\in F_n$, there exists $\wb\in F_{2r}$ such that $\vb\in F_{r,n}(\wb)$. We have $\langle \tilde{\ub},\wb\rangle\ge 0$ because $\tilde{\ub}\in C_{2r}$ and $\wb\in F_{2r}\subset C_{2r}^*$. Applying \Cref{l:padding} to $\tilde{\ub}$ and $\wb$, noting that $\tilde{u}_r\ge 0\ge \tilde{u}_{r+1}$ and that $\vb\in F_{r,n}(\wb)\subseteq \OO_{r,n}(\wb)$,  we obtain 
	\[
	\langle \sigma(\ub),\vb\rangle
	=\langle \sigma\circ\tau^{-1}(\tilde{\ub}),\vb\rangle
	\ge 0,
	\]
	as desired.
	
	To complete the proof, it remains to show that $C_n^*\subseteq\cone(\Sym(n)(F_n))$ for all $n>2r$. Since $C_n^*=\Sym(n)(\OO(C_n^*))$, it is enough to prove the inclusion $\OO(C_n^*)\subseteq\cone(F_n)$.
	Take any $\bb=(b_1,\dots,b_n)\in \OO(C_n^*)$. Then repeated applications of \Cref{l:dual-cone} yield
	\[
	\cb\defas(b_1,\dots,b_r,b_{n-r+1},\dots,b_n)\in \OO(C_{2r}^*) =\cone(F_{2r}),
	\]
	where the last equality follows from the assumption that $F_{2r}$ generates $\OO(C_{2r}^*)$.
	Evidently, $\bb\in \OO_{r,n}(\cb)$, and thus $\bb\in \conv(F_{r,n}(\cb))$ by \Cref{l:convex-hull}. According to \Cref{l:Fim}, we have 
	$$
	F_{r,n}(\cb)
	\subseteq\cone\Big(\bigcup_{\ab\in F_{2r}} F_{r,n}(\ab)\Big)
	=\cone(F_n).
	$$
	Consequently,
	\begin{align*}
		\bb\in 
		\conv(F_{r,n}(\cb))
		\subseteq
		\conv(\cone(F_n))
		=\cone(F_n).
	\end{align*}
	Therefore, $\OO(C_n^*)\subseteq\cone(F_n)$, as required.
\end{proof}

\begin{remark}
	\label{r:W-M-improved}
 \Cref{t:Minkowski-Weyl} can be slightly refined to encompass the case of nonnegative cones treated in \cite[Theorem 4.4]{LR21}. Adopting the notation from the proof of \Cref{t:Minkowski-Weyl}, we define
 \begin{align*}
 s&=\max\{i\in[r]\mid u_i>0\ \text{ for some } \ (u_1,\dots,u_r)\in G_r\},\\
 t&=\max\{j\in[r]\mid u_{r-j}<0\ \text{ for some } \ (u_1,\dots,u_r)\in G_r\},
 \end{align*}
 with the convention that $s=0$ if $G_r\subseteq\RR_{\le0}^r$ and $t=0$ if $G_r\subseteq\RR_{\ge0}^r$.
 Let $p=\max\{s+t,r+1\}$. Then $r+1\le p\le 2r$, and $p=r+1$ if $G_r\subseteq\RR_{\ge0}^r$ or $G_r\subseteq\RR_{\le0}^r$. Suppose $F_p \subseteq \OO(\RR^p)$ is a generating set for $\OO(C_p^*)$. Then by adapting the argument used in the proof of \Cref{t:Minkowski-Weyl}, one can show that for all $n>p$, the set
 $$F_n\defas\bigcup_{\ub\in F_{p}} F_{s,n}(\ub)$$
 is a $\Sym(n)$-equivariant generating set for $C_n^*$. This result generalizes \cite[Theorem 4.4]{LR21}, because in the context of nonnegative cones, it can be  verified that given a generating set $F_{r}\subseteq \OO(\RR^{r})$ for $\OO(C_{r}^*)$, the set
 \[
 F_{r+1}=\{(u_1,\dots,u_r,u_r)\mid\ub=(u_1,\dots,u_r)\in F_r,\ |\supp(\ub)|\ge2\}\cup\{\eb_{r+1}\}
 \subseteq \OO(\RR^{r+1})
 \]
 is a generating set for $\OO(C_{r+1}^*)$.
\end{remark}

Given a $\Sym$-invariant cone $C\subseteq\RR^{(\NN)}$, \Cref{t:Minkowski-Weyl} has implications for the global dual cone $C^*$. When $C\ne\{\nub\}$, identifying an explicit nonzero element of $C^*$ is generally nontrivial. However, as a consequence of \Cref{t:Minkowski-Weyl}, we obtain the following result. For $\ub=(u_1,\dots,u_{n})\in\RR^n$ and $\vb=(v_1,v_2,\dots)\in \RR^{\NN}$, let $(\ub,\vb)$ denote the element of $\RR^{\NN}$ formed by concatenating $\ub$ and $\vb$:
\[
(\ub,\vb)=(u_1,\dots,u_{n},v_1,v_2,\dots).
\] 

\begin{corollary}
	\label{c:global-dual-cone}
	Let $C\subseteq\RR^{(\NN)}$ be a $\Sym$-invariant cone and let $C^*\subseteq\RR^{\NN}$ denote its dual cone. Suppose that $C$ is $\Sym$-equivariantly generated by a subset $G_r\subseteq\RR^r$ for some $r\ge1$. Let $C_{2r}^*\subseteq\RR^{2r}$ be the dual cone of the cone $C_{2r}=	\cone(\Sym(2r)(G_r))$. Then the set
	\[
	F=\{(\ub,\vb)\mid\ub=(u_1,\dots,u_{2r})\in \OO(C_{2r}^*),\ \vb\in \OO([u_r,u_{r+1}]^\NN)\}
	\]
	is contained $C^*$.
\end{corollary}

\begin{proof}
	Consider the chain of cones $\C=(C_{n})_{n\geq 1}$ defined by
	\[
	C_n=
	\begin{cases}
		\{\nub\} &\text{if } n<r,\\
		\cone(\Sym(n)(G_r))&\text{if } n\ge r.
	\end{cases}
	\] 
	Evidently, $\C$ is a stabilizing $\Sym$-invariant chain with stability index $r.$ Let $\C^*=(C_{n}^*)_{n\geq 1}$ denote the corresponding chain of dual cones.  Since $\OO(C_{2r}^*)$ is trivially a generating set for $\OO(C_{2r}^*)$, it follows from \Cref{t:Minkowski-Weyl} that
	\[
	F_n=\bigcup_{\ub\in \OO(C_{2r}^*)} F_{r,n}(\ub)
	\]
	is a $\Sym(n)$-equivariant generating set for $C_n^*$ for all $n>2r$. 
	
	In order to show that $F\subseteq C^*$, it suffices to verify that
	\[
	\langle \sigma(\ab),\wb\rangle\ge 0
	\  \text{ for all }\ \ab\in G_r, \ \wb\in F, \ \sigma\in\Sym.
	\]
	Choose $n>2r$ such that $\sigma\in\Sym(n)$. Then $\sigma(\ab)\in C_n\subseteq\RR^n$, and thus
	\[
	\langle \sigma(\ab),\wb\rangle
	=\langle \sigma(\ab),\wb_n\rangle,
	\]
	where $\wb_n\in \RR^n$ denotes the truncation of $\wb$ to its first $n$ coordinates. By the definition of $F$, we have $\wb_n=(\ub,\vb_{n-2r})$ for some $\ub\in \OO(C_{2r}^*)$ and $\vb_{n-2r}\in \OO([u_r,u_{r+1}]^{n-2r})$. Consequently, there exists $\tau\in\Sym(n)$ such that
	\[
	\tau(\wb_n)=(u_1,\dots,u_r,\vb_{n-2r},u_{r+1},\dots,u_{2r})
	\in \OO_{r,n}(\ub).
	\]
	It thus follows from \Cref{l:convex-hull} that
	\[
	\tau(\wb_n)\in\conv(F_{r,n}(\ub))
	\subseteq\conv(F_{n})
	\subseteq\cone(F_{n})
	\subseteq C_n^*.
	\]
	This implies $\wb_n\in C_n^*$ since $C_n^*$ is $\Sym(n)$-invariant. Therefore,
	\[
	\langle \sigma(\ab),\wb\rangle
	=\langle \sigma(\ab),\wb_n\rangle\ge 0,
	\]
	as desired.
\end{proof}

\begin{example}
Let $C\subseteq\RR^{(\NN)}$ be the cone $\Sym$-equivariantly generated by the two vectors:
\[
(-2,-1,4)\quad\text{and}\quad (-3,1,3).
\]
Then $C$ is the direct limit of the chain $\C$ considered in \Cref{e:dual-cone}. From this example and \Cref{c:global-dual-cone}, we obtain the following elements of the dual cone $C^*$:
\[
(1,1,1,\dots),\ (5,6^{(4)},7,6,6,6,\dots),\ (3^{(i)},4^{(6-i)},\vb)
\]
for all $i\in[5]$ and $\vb\in\OO([3,4]^\NN)$.
\end{example}

\Cref{t:Minkowski-Weyl} establishes a necessary condition for the stabilization of a $\Sym$-invariant chain of cones $\C$ in terms of its dual chain $\C^*$. It remains an open question whether this condition is also sufficient. More broadly, the following problem merits further investigation.

\begin{problem}
	Characterize the stabilization, and more generally, other properties of interest (such as equivariant finite generation), of a $\Sym$-invariant chain of cones $\C$ in terms of its dual chain $\C^*$. 
\end{problem}

\section{Local-global principles}
\label{section-local-global}

In this section, we characterize equivariant finite generation of $\Sym$-invariant cones in terms of their associated chains of local cones. We also present a generalization of this result that characterizes the stabilization of such chains. As a consequence, we derive a characterization of equivariant finite generation for $\Sym$-invariant normal monoids. These results, which are of independent interest, play a crucial role in our proof of the equivariant Gordan's lemma in \Cref{sec:equi-Gordan}. 

We first establish the following characterization of equivariant finite generation for global cones, extending the case of nonnegative cones considered in \cite[Corollary 5.5]{KLR} and \cite[Lemma 2.1]{LR21}. An analogous result for lattices can be found in \cite[Theorem 3.6]{LR24}.

\begin{theorem}
	\label{t:finte-generation}
	Let $\C=(C_{n})_{n\geq 1}$ be a $\Sym$-invariant chain of cones with limit $C=\bigcup_{n\geq 1}C_n.$ Then the following statements are equivalent:
	\begin{enumerate}
		\item
		$\C$ stabilizes, and $C_n$ is eventually finitely generated;
		\item 
		There exist $p\in\NN$ and a finite subset $G_p\subseteq \RR^{p}$ such that $G_p$ is a $\Sym(n)$-equivariant generating set for $C_n$ for all $n\ge p$;
		\item
		There exists $q\in\NN$ such that for all $n\ge q$, the following conditions hold:
		\begin{enumerate}
			\item
			$C \cap \RR^{n}= C_{n}$,
			\item
			$C_{n}$ is finitely generated by elements of support size at most $q$;
		\end{enumerate}
	\item
	$C$ is $\Sym$-equivariantly finitely generated.
	\end{enumerate}
\end{theorem}

\begin{proof}
	(i)$\Rightarrow$(ii): Assume that $\C$ stabilizes with stability index $r$. Choose $p\ge r$ such that $C_p$ has a finite generating set $G_p$. Since the chain stabilizes at $r$, we have
	\[
	C_n=\cone(\Sym(n)(C_p))
	=\cone(\Sym(n)(G_p))
	\quad\text{for all }\ n\ge p.
	\]
	Hence, $G_p$ is a $\Sym(n)$-equivariant generating set for $C_n$ for all $n\ge p$. 
	\smallskip
	
	(ii)$\Rightarrow$(iii): We show that the two conditions in (iii) hold with $q=2p$. By assumption, $C_n$ is generated by $G_n\defas\Sym(n)(G_p)$ for all $n\ge p$. Since $G_p$ is a finite subset of $\RR^p$, it is evident that $G_n$ is also finite and its elements have support size at most $p$ for all $n\ge p$. In particular, every element of $G_n$ has support size at most $q$ for all $n\ge q$, establishing (iii)(b).
	
	To verify (iii)(a), it suffices to show that $C_{n+1}\cap\RR^n=C_n$ for all $n\ge q$. Indeed, this implies
	\[
	C_m\cap\RR^n=(C_m\cap\RR^{m-1})\cap\RR^n
	=C_{m-1}\cap\RR^n=\cdots
	=C_{n+1}\cap\RR^n=C_n
	\]
	for all $m>n\ge q$, and hence
	\[
	C \cap \RR^{n}= \bigcup_{m\geq 1}C_m\cap\RR^n
	=\bigcup_{m\geq n+1}(C_m\cap\RR^n)
	=C_{n}
	\]
	for all $n\ge q$, as desired.
	
	To prove that $C_{n+1}\cap\RR^n=C_n$ for all $n\ge q$, we use two key observations: 
	\begin{itemize}
	\item 
	Since $C_n$ is finitely generated, it satisfies $C_n=C_n^{**}$, where $C_n^{**}$ denotes the double dual cone of $C_n$ (see, e.g., \cite[Theorem 1.16(b)]{BG}). 
	\item 
	The chain $\C$ stabilizes with stability index $r \leq p=q/2$, and thus, by \Cref{t:Minkowski-Weyl} and \Cref{r:M-W}(i), there exists a $\Sym(2r)$-equivariant generating set $F_{2r}$ for $C_{2r}^*$ that determines a $\Sym(n)$-equivariant generating set $F_n$ for $C_n^*$ for all $n \ge q$. 
	\end{itemize}
	Now let $\ub=(u_1,\dots,u_n)\in C_{n+1}\cap\RR^n$ for some $n\ge q$. 
	To show that $\ub\in C_n=C_n^{**}$, we need to verify that $\langle \ub,\vb\rangle\ge 0$ for all $\vb\in C_n^*$. Since $C_n^*$ is $\Sym(n)$-equivariantly generated by $F_n$, it is enough to check that $\langle \ub,\sigma(\wb)\rangle\ge 0$ for all $\wb=(w_1,\dots,w_n)\in F_n$ and all $\sigma\in\Sym(n).$ 
	From the construction of $F_n$ and $F_{n+1}$, it follows that $$\widetilde{\wb}\defas(w_1,\dots,w_r,w_{r},w_{r+1},\dots,w_n)\in F_{n+1}.$$
	Choose  $\tau\in\Sym(n+1)$ such that $\tau(\widetilde{\wb})=(\sigma(\wb),w_r).$ 
	We have $\langle \ub,\tau(\widetilde{\wb})\rangle\ge0$ because $\ub\in C_{n+1}$ and $\tau(\widetilde{\wb})
	\in \Sym(n+1)(F_{n+1})\subseteq C_{n+1}^*.$
	Since 
	$$
	0\le\langle \ub,\tau(\widetilde{\wb})\rangle
	=\langle \ub,\sigma(\wb)\rangle+0\cdot w_s
	=\langle \ub,\sigma(\wb)\rangle,
	$$
	we are done.
	\smallskip
	
	(iii)$\Rightarrow$(i): It remains to show that $\C$ stabilizes. Let $n\ge q$ and let $G_n$ be a generating set for $C_n$ whose elements have support size at most $q$. Then for each $\ub\in G_n$, there exists a $\sigma\in\Sym(n)$ such that $\ub'=\sigma(\ub)\in \RR^q$. Since $\ub=\sigma^{-1}(\ub')$ and
	$\ub'\in C\cap \RR^q=C_q,$
	we deduce that $G_n\subseteq\Sym(n)(C_q).$ Therefore,
	\[
	C_n=\cone(G_n)
	\subseteq\cone(\Sym(n)(C_q))
	\subseteq C_n,
	\] 
	and consequently, $C_n=\cone(\Sym(n)(C_q))$ for all $n\ge q$, as required.
	\smallskip
	
	(ii)$\Rightarrow$(iv): The set $G_p$ clearly forms a $\Sym$-equivariant generating set for $C=\bigcup_{n\ge1}C_n.$ 
	\smallskip
	
	(iv)$\Rightarrow$(iii): Let $G$ be a finite $\Sym$-equivariant generating set for $C$. Since $C=\bigcup_{n\ge1}C_n$, there exists $p\ge 1$ such that $G\subseteq C_p$. Define a new chain $\C'=(C_{n}')_{n\geq 1}$ with
	\[
	C_{n}'=
	\begin{cases}
		\{\nub\}&\text{if } 1\le n< p,\\
		\cone(\Sym(n)(G))&\text{if } n\ge p.
	\end{cases}
	\]
	Evidently, $\C'$ is a $\Sym$-invariant chain whose limit is $C = \bigcup_{n \geq 1} C_n'$, and $G$ serves as a $\Sym(n)$-equivariant generating set for $C_n'$  for all $n\ge p$. By the implication (ii)$\Rightarrow$(iii) already proved, we have $C\cap \RR^n=C_{n}'$ for all $n\ge 2p$. Furthermore, since
	\[
	C_{n}'=\cone(\Sym(n)(G))\subseteq\cone(\Sym(n)(C_p))\subseteq C_n\subseteq C\cap \RR^n
	\]
	for all $n\ge p$, we conclude that
	\[
	C_n=\cone(\Sym(n)(G))= C\cap \RR^n
	\ \text{ for all } n\ge 2p.
	\]
	This proves (iii) and thus completes the proof of the theorem.
\end{proof}

\begin{remark}
	As we have seen in the proof of the implication (ii)$\Rightarrow$(iii) of \Cref{t:finte-generation}, the property of the chain of dual cones established in \Cref{t:Minkowski-Weyl} is instrumental in proving the equality $C \cap \RR^n = C_n$ for sufficiently large $n$. In the case of nonnegative cones, however, the argument simplifies significantly and relies essentially on the following elementary fact: if $\ub = \sum_{i=1}^s \lambda_i \ab_i$ with $\lambda_i > 0$ and $\ab_i \in \RR^{(\NN)}_{\geq 0}$, then $\supp(\ab_i) \subseteq \supp(\ub)$ for all $i \in [s]$. For further details, see the proof of \cite[Corollary 2.2]{LR21}.
\end{remark}

The next result provides several characterizations of the stabilization of $\Sym$-invariant chains of local cones. One characterization generalizes and refines the case of nonnegative cones treated in \cite[Lemma 6.2]{LR21}, while the remaining characterizations are obtained by removing the finite generation conditions from \Cref{t:finte-generation}.

\begin{theorem}
	\label{t:stabilizing-cones}
	Let $\C=(C_{n})_{n\geq 1}$ be a $\Sym$-invariant chain of cones with limit $C=\bigcup_{n\geq 1}C_n.$ Then the following statements are equivalent:
	\begin{enumerate}
		\item
		$\C$ stabilizes;
		\item 
		There exist $p\in\NN$ and a  subset $G_p\subseteq \RR^{p}$ such that $G_p$ is a $\Sym(n)$-equivariant generating set for $C_n$ for all $n\ge p$;
		\item
		There exists $q\in\NN$ such that for all $n\ge q$, the following conditions hold:
		\begin{enumerate}
			\item
			$C \cap \RR^{n}= C_{n}$,
			\item
			For any $(u_1,\dots,u_{n+1})\in C_{n+1}$, there exists $\sigma\in \Sym(n+1)$ such that 
			\[
			u_{\sigma(n)}u_{\sigma(n+1)}\ge0
			\quad\text{and}\quad
			(u_{\sigma(1)},\dots,u_{\sigma(n-1)},u_{\sigma(n)}+u_{\sigma(n+1)})
			\in C_{n};
			\]
		\end{enumerate}
		\item
		There exists $q\in\NN$ such that for all $n\ge q$, the following conditions hold:
		\begin{enumerate}
			\item
			$C \cap \RR^{n}= C_{n}$,
			\item
			$C_{n}$ is generated by elements of support size at most $q$.
		\end{enumerate}
		\item
		$C$ is $\Sym$-equivariantly generated by a subset of $C_{r}$ for some  $r\in\NN$.
	\end{enumerate}
\end{theorem}

\enlargethispage{\baselineskip}

Given \Cref{t:finte-generation}, the main task in the proof of the theorem above is to establish the equivalence between statement (iii) and the others. To this end, the following preparatory result is needed.

\begin{lemma}
	\label{l:support-reduction}
	Let $p\ge 1$ and let $C_n\subseteq\RR^n$ be a $\Sym(n)$-invariant cone generated by elements of support size at most $p$. If $n\ge3p^2$, then for any $\ub=(u_1,\dots,u_n)\in C_n$, there exists $\sigma\in \Sym(n)$ such that 
	\[
	u_{\sigma(n-1)}u_{\sigma(n)}\ge0
	\quad\text{and}\quad
	(u_{\sigma(1)},\dots,u_{\sigma(n-2)},u_{\sigma(n-1)}+u_{\sigma(n)})
	\in C_{n}.
	\]
\end{lemma}

\begin{proof}
	By Carath\'{e}odory's theorem (see, e.g., \cite[Theorem 1.55]{BG}), we can write $\ub=\sum_{i=1}^n\lambda_i \vb_i$ with $\lambda_i\ge0$, $\vb_i\in C_n$, and moreover,  $|\supp(\vb_i)|\le p$ for all $i\in[n]$ by assumption. For each $j\in[n]$, set
	\[
	T_j=\{i\in[n]\mid j\in\supp(\vb_i)\}
	\quad\text{and}\quad
	S_j=\bigcup_{i\in T_j}\supp(\vb_i).
	\]
	Then
	\begin{equation}
		\label{e:Sj}
		|S_{j}|\le \sum_{i\in T_{j}}|\supp(\vb_i)|\le p|T_j|
		\ \text{ for all } j\in[n].
	\end{equation}
	Let us first prove the following:
	
	\begin{claim}
		\label{cl:notin}
		Up to a permutation of the coordinates of $\ub$, we have
		\[
		n-1\notin S_{n}
		\ \text{ and }\ 
		 n-2\notin S_{n-1}\cup S_{n}.
		\]
	\end{claim}
	
	\begin{proof}[Proof of \Cref{cl:notin}]
		By double counting, we obtain
		\begin{equation}
			\label{e:cardinality}
			\sum_{j=1}^n|T_j|=\sum_{i=1}^n|\supp(\vb_i)|\le np.
		\end{equation}
		Permuting the coordinates of $\ub$, if necessary, we may assume that $T_{n}$ has minimal cardinality among the sets $T_j$ for $j\in[n]$. Then \eqref{e:cardinality} yields $|T_{n}|\le p.$ It thus follows from \eqref{e:Sj} that
		$
		|S_{n}|\le  p^2.
		$
		Hence, $[n]\setminus S_{n}\ne \emptyset$ since $n\ge3p^2$. Permuting coordinates again if needed, we may assume that $n-1\in [n]\setminus S_{n}$ and that $T_{n-1}$ has minimal cardinality among the sets $T_j$ with $j\in[n]\setminus S_{n}$. Then
		\[
		(n-p^2)|T_{n-1}|\le (n-|S_{n}|)|T_{n-1}|\le \sum_{j\in [n]\setminus S_{n}}|T_j|
		\le \sum_{j=1}^n|T_j|\le np.
		\]
		Since $n\ge3p^2$, it is easy to check that
		\[
		|T_{n-1}|\le\frac{np}{n-p^2}< 2p.
		\]
		This together with \eqref{e:Sj} implies that
		$
		|S_{n-1}|< 2p^2.
		$
		Therefore,
		\[
		|S_{n-1}\cup S_{n}|\le |S_{n-1}|+|S_{n}|< 2p^2+p^2=3p^2\le n.
		\]
		Thus, there exists $k\in[n]\setminus (S_{n-1}\cup S_{n})$, and by permuting coordinates once more, we may assume that $k=n-2$. This completes the proof of the claim.
	\end{proof}
	
	\begin{claim}
		\label{cl:nonnegative}
		Up to a permutation of the coordinates of $\ub$, it holds that $n-1\notin S_n$ and $u_{n-1}u_n\ge0$.
	\end{claim}
	
	\begin{proof}[Proof of \Cref{cl:nonnegative}]
		By \Cref{cl:notin}, we may assume that $n-1\notin S_{n}$ and $n-2\notin S_{n-1}\cup S_{n}$. If $u_{n-1}u_n\ge0$, we are done. Otherwise, either $u_{n-2}u_n\ge0$ or $u_{n-2}u_{n-1}\ge0$ must hold. Since $n-2\notin S_n$ and $n-2\notin S_{n-1}$, the desired conclusion results from transposing $u_{n-2}$ with either $u_{n-1}$ or $u_{n}$.
	\end{proof}
	
	Let us resume the proof of the lemma. By \Cref{cl:nonnegative}, we may assume that $n-1\notin S_n$ and $u_{n-1}u_n\ge 0$. So it remains to show that
	\[
	(u_{1},\dots,u_{n-2},u_{n-1}+u_{n})
	\in C_{n}.
	\]
	Recall that 
	\[
	\ub=\sum_{i=1}^n\lambda_i \vb_i
	=\sum_{i\in T_n}\lambda_i\vb_i+\sum_{i\in [n]\setminus T_{n}}\lambda_i\vb_i.
	\]
	Since $n\notin\supp(\vb_i)$ for all $i\in[n]\setminus T_{n}$ and $n-1\notin\supp(\vb_i)$ for all $i\in T_{n}$, we can write
	\[
	\sum_{i\in T_n}\lambda_i\vb_i=(u_1',\dots,u_{n-2}',0,u_n)
	\quad\text{and}\quad
	\sum_{i\in [n]\setminus T_{n}}\lambda_i\vb_i=(u_1'',\dots,u_{n-2}'',u_{n-1},0),
	\]
	where $u_j=u_j'+u_j''$ for all $j\in[n-2]$.
	Let $\pi\in \Sym(n)$ be the transposition that swaps $n-1$ and $n$. Then
	\[
	(u_{1},\dots,u_{n-2},u_{n-1}+u_{n},0)
	=\sum_{i\in T_n}\lambda_i\pi(\vb_i)+\sum_{i\in [n]\setminus T_{n}}\lambda_i\vb_i.
	\]
	Since $C_n$ is $\Sym(n)$-invariant, it follows that
	\[
	(u_1, \dots, u_{n-2}, u_{n-1}+u_n) \in C_n,
	\]
	as required.
\end{proof}

Let us now prove \Cref{t:stabilizing-cones}.

\begin{proof}[Proof of \Cref{t:stabilizing-cones}]
	We only prove the implications (ii)$\Rightarrow$(iii)$\Rightarrow$(iv).  
	The implications (i)$\Rightarrow$(ii)$\Rightarrow$(v)$\Rightarrow$(iv)$\Rightarrow$(i) can be established similarly to the proof of \Cref{t:finte-generation}; the details are left to the reader.
	
	(ii)$\Rightarrow$(iii): Suppose $G_p\subseteq \RR^{p}$ is a $\Sym(n)$-equivariant generating set for $C_n$ for all $n\ge p$. We show that the two conditions in (iii) hold with $q=3p^2$. To prove (iii)(a), let $n\ge 3p^2$ and take $\ub\in C\cap\RR^n$. Then $\ub\in C_m$ for some $m\ge n$, and we can write
	\[
	\ub=\sum_{i=1}^t\lambda_i \sigma_i(\vb_i),
	\quad\text{where $\lambda_i\ge0, \ \vb_i\in G_p$, and $\sigma_i\in\Sym(m)$ for all $i\in[t]$.} 
	\]
	Set $G_p'=\{\vb_1,\dots,\vb_t\}$. Consider the $\Sym$-invariant chain $\C'=(C_{k}')_{k\geq 1}$ defined by
	\[
	C_{k}'=
	\begin{cases}
		\{\nub\}&\text{if } 1\le k< p,\\
		\cone(\Sym(k)(G_p'))&\text{if } k\ge p,
	\end{cases}
	\]
	and denote  $C'=\bigcup_{k\ge1}C_k'$.
	Then $G_p'$ is a finite $\Sym(k)$-equivariant generating set for $C_k'$ for all $k\ge p$. From the proof of the implication (ii)$\Rightarrow$(iii) in \Cref{t:finte-generation}, it follows that $C'\cap\RR^k=C_k'$ for all $k\ge 2p$. In particular, since $n\ge 3p^2>2p$, we have $C'\cap\RR^n=C_n'$.  Since $\ub\in C'\cap\RR^n$ and $C_n'\subseteq C_n$, we deduce that $\ub\in C_n$. Therefore, $C\cap\RR^n=C_n$, as required.
	
	To prove (iii)(b), we first note that $C_{n+1}\cap\RR^n=C_n$ for all $n\ge 3p^2$. Indeed, this follows from (iii)(a) and the obvious inclusions $C_n\subseteq C_{n+1}\cap\RR^n\subseteq C\cap\RR^n.$  Now \Cref{l:support-reduction} yields the desired conclusion.
	
	(iii)$\Rightarrow$(iv): We need to show that any element $\ub\in C_n$ with support size $s>q$ can be generated by elements of $C_n$ with smaller support sizes. Permuting the coordinates of $\ub$, if necessary, we may assume that $\ub=(u_1,\dots,u_s)$, where $\ub_i\ne0$ for all $i\in[s].$ Then we have $\ub\in C\cap\RR^s=C_s$ since $s>q$. By (iii)(b), up to a permutation of the coordinates of $\ub$, it can be assumed that $u_{s-1}u_s>0$ and $\vb=(u_{1},\dots,u_{s-2},u_{s-1}+u_{s})\in C_{s-1}$. The latter implies
	$$
	\wb=(u_{1},\dots,u_{s-2},0,u_{s-1}+u_{s})\in C_s.
	$$
	Observe that $\vb,\wb\in C_n$ (since $s=|\supp(\ub)|\le n$) and both of them have support size $s-1$. 
	Evidently, $\ub=\lambda\vb+\mu\wb$, where $\lambda=u_{s-1}/(u_{s-1}+u_s)$ and $\mu=u_{s}/(u_{s-1}+u_s)$ are positive numbers, because $u_{s-1}u_s>0$. This shows that $\ub$ is generated by elements with smaller support size, completing the proof.
\end{proof}

\Cref{t:finte-generation,t:stabilizing-cones} yield the following characterization of equivariant finite generation of positive normal $\Sym$-invariant monoids. In the case of nonnegative monoids, similar characterizations have been established in \cite[Corollary 5.13]{KLR}, \cite[Lemmas 5.1 and 6.5]{LR21}, and \cite[Theorem 3.14]{LR24}. For any element $\ub=(u_i)_{\ib\in \NN}\in \RR^{(\NN)}$ and any subset $A\subseteq\RR^{(\NN)}$, we define
\[
\|\ub\|=\sum_{\ib\in \NN}|u_i|
\quad\text{and}\quad
\|A\|=\sup\{\|\ub\|: \ub\in A\}.
\]

\begin{corollary}
	\label{c:finte-generation-monoid}
	Let $\M=(M_{n})_{n\geq 1}$ be a $\Sym$-invariant chain of positive normal affine monoids with limit $M=\bigcup_{n\geq 1}M_n.$ For each $n\ge1$, let $\Hc_{n}$ denote the Hilbert basis of $M_n$. Then the following statements  are equivalent:
	\begin{enumerate}
		\item
		$\M$ stabilizes;
		\item 
		There exists $r\in\NN$ such that $\Hc_r$ is a $\Sym(n)$-equivariant Hilbert basis for $M_n$ for all $n\ge r$;
		\item
		There exists $p\in\NN$ such that for all $n\ge p$, the following conditions hold:
		\begin{enumerate}
			\item
			$M \cap \ZZ^{n}= M_{n}$,
			\item
			$\|\Hc_m\|\le \|\Hc_n\|$ for all $m\ge n\ge p$;
		\end{enumerate}
		\item
		There exists $p\in\NN$ such that for all $n\ge p$, the following conditions hold:
		\begin{enumerate}
			\item
			$M \cap \ZZ^{n}= M_{n}$,
			\item
			For any $(u_1,\dots,u_{n+1})\in M_{n+1}$, there exists $\sigma\in \Sym(n+1)$ such that 
			\[
			u_{\sigma(n)}u_{\sigma(n+1)}\ge0
			\quad\text{and}\quad
			(u_{\sigma(1)},\dots,u_{\sigma(n-1)},u_{\sigma(n)}+u_{\sigma(n+1)})
			\in M_{n};
			\]
		\end{enumerate}
		\item
		There exists $q\in\NN$ such that for all $n\ge q$, the following conditions hold:
		\begin{enumerate}
			\item
			$M \cap \ZZ^{n}= M_{n}$,
			\item
			$M_{n}$ is generated by elements of support size at most $q$;
		\end{enumerate}
		\item
		$M$ is $\Sym$-equivariantly finitely generated.
	\end{enumerate}
\end{corollary}

To prove this corollary, we recall the following auxiliary result from \cite[Corollary 6.4]{LR21}.

\begin{lemma}
	\label{l:between-monoid}
	Let $M_n\subseteq\RR^n$ be a $\Sym(n)$-invariant normal monoid. If $(u_1,\dots,u_n)\in M_n$ and $v_{n-1},v_{n}$ lie in the interval with the endpoints $u_{n-1}$ and $u_n$ such that $v_{n-1}+v_{n}=u_{n-1}+u_n$, then $(u_1,\dots,u_{n-2},v_{n-1},v_n)\in M_n.$
\end{lemma}

\begin{proof}[Proof of \Cref{c:finte-generation-monoid}]
Let $C=\cone(M)$ and $C_n=\cone(M_n)$ for $n\ge 1.$ Then $\C=(C_{n})_{n\geq 1}$ forms a $\Sym$-invariant chain of cones whose limit is $C$. Moreover, since $M_n$ is a normal monoid, we have $M_n=C_n\cap \ZZ^{n}$ for $n\ge1$. It follows that  $M=C\cap \ZZ^{(\NN)}$.
\smallskip

(vi)$\Rightarrow$(iv): Suppose $M$ is $\Sym$-equivariantly finitely generated. Then the cone $C=\cone(M)$ is also $\Sym$-equivariantly finitely generated. By the equivalence of statements (iii) and (v) in \Cref{t:stabilizing-cones}, it follows that $C\cap\RR^n=C_n$ for all $n\gg0$. Therefore,
\[
M \cap \ZZ^{n}=C\cap  \ZZ^{n}= C_n\cap  \ZZ^{n}=M_{n}
\]
for all $n\gg0$, which establishes part (iv)(a). Part (iv)(b) follows readily from statement (iii)(b) of \Cref{t:stabilizing-cones}, together with the fact that $M_n=C_n\cap \ZZ^{n}$ for all $n\ge1$.
\smallskip

(iv)$\Rightarrow$(iii): It is enough to show that $\|\Hc_{n+1}\|\le \|\Hc_n\|$ for all $n\ge p$. Since $M_{n+1}$ is $\Sym(n+1)$-invariant, so is its Hilbert basis $\Hc_{n+1}$. Let $\ub=(u_1,\dots,u_{n+1})\in \Hc_{n+1}$ be such that $\|\ub\|=\|\Hc_{n+1}\|.$ By permuting the coordinates of $\ub$, we may assume that $u_nu_{n+1}\ge0$ and 
$\ub'\defas(u_1,\dots,u_{n-1},u_n+u_{n+1})\in M_n.$
Note that $\|\ub'\|=\|\ub\|$ since $u_nu_{n+1}\ge0$. So to conclude the implication, it suffices to show that $\ub'\in \Hc_n$. Suppose, for the sake of contradiction, that $\ub'=\vb'+\wb'$ for some $\vb', \wb'\in M_n\setminus\{\nub\}$. Write $\vb'=(v_1,\dots,v_n)$ and $\wb'=(w_1,\dots,w_n)$. Then $u_n+u_{n+1}=v_n+w_n$, and the condition $u_nu_{n+1}\ge0$ implies
\[
|u_n|+|u_{n+1}|=|u_n+u_{n+1}|=|v_n+w_n|\le |v_n|+|w_{n}|.
\]
Without loss of generality, suppose that $|u_n|\le|u_{n+1}|$ and $|w_{n}|\le |v_n|$. Then $|u_{n}|\le |v_n|$. It must also hold that $u_nv_n\ge0$, because
\[
u_nv_n+u_nw_n=u_n(u_n+u_{n+1})\ge0
\quad\text{and}\quad
|u_nw_{n}|\le |u_nv_n|.
\]
Thus, both $u_n$ and $v_n-u_n$ lie between $0$ and $v_n$. Since $\vb'=(v_1,\dots,v_n,0)\in M_{n+1}\setminus\{\nub\}$, it follows from \Cref{l:between-monoid} that $\vb\defas(v_1,\dots,v_{n-1},u_n,v_n-u_n)\in M_{n+1}\setminus\{\nub\}.$ Moreover, we have $\wb\defas(w_1,\dots,w_{n-1},0,w_n)\in M_{n+1}\setminus\{\nub\}$ because $\wb'=(w_1,\dots,w_n,0)\in M_{n+1}\setminus\{\nub\}.$ Finally, since $\ub'=\vb'+\wb'$, we obtain $\ub=\vb+\wb$, contradicting the fact that $\ub\in \Hc_{n+1}.$ Therefore, $\ub'\in \Hc_n$, and the desired inequality follows.
\smallskip

(iii)$\Rightarrow$(v): Let $q=\max\{p,\|\Hc_p\|\}$. Then for any $\ub\in \Hc_n$ with $n\ge q$, we have
\[
|\supp(\ub)|\le \|\ub\|\le \|\Hc_n\|\le \|\Hc_p\|\le q.
\] 
Therefore, $M_{n}$ is generated by elements of support size at most $q$ for all $n\ge q$.
\smallskip

(v)$\Rightarrow$(i): This implication follows by an argument analogous to that used in the proof of (iii)$\Rightarrow$(i) in \Cref{t:finte-generation}, and we omit the details here.
\smallskip

(i)$\Rightarrow$(ii):  This is immediate from the definition of stabilization.
\end{proof}

To conclude this section, we propose the following problem, which points to a promising direction for future research.

\begin{problem}
	\label{pb:local-global-monoids}
	Investigate the extent to which \Cref{c:finte-generation-monoid} can be generalized to $\Sym$-invariant chains of monoids that are not necessarily positive or normal.
\end{problem}


\section{Non-pointed symmetric cones and non-positive symmetric monoids}
\label{sec-non-pointed-cones}

In this section, we continue the preparatory work for the proof of the equivariant Gordan's lemma by classifying non-positive $\Sym$-invariant normal monoids, thereby complementing the treatment of positive normal monoids given in \Cref{c:finte-generation-monoid}. This result is derived from a classification of non-pointed $\Sym$-invariant cones, which is also of independent interest. Local versions of both classification results are included, providing further instances where local and global settings exhibit distinct behavior.

\subsection{Non-pointed symmetric cones}

The \emph{lineality space} of a cone $C\subseteq\RR^{(\NN)}$ is defined as
\[
\lin(C)=\{\ub\in C\mid -\ub\in C\}.
\]
This is the largest vector subspace contained in $C$. We say that $C$ is \emph{pointed} if $\lin(C)=\{0\}.$ 
Our goal here is to classify all non-pointed $\Sym$-invariant cones in $\RR^{(\NN)}$. To state the result, we introduce the following notation:
for any $\ub=(u_i)_{i\in\NN}\in\RR^{(\NN)}$, let 
$s(\ub)=\sum_{i\in\NN}u_i$ denote the sum of its entries. Note that $s(\ub)$ is well-defined since $\ub$ has finite support. 

\begin{theorem}
	\label{t:non-pointed-cones}
	There are precisely four non-pointed $\Sym$-invariant cones in $\RR^{(\NN)}$, namely,
	\begin{align*}
		\mathfrak{C}_1&=\RR^{(\NN)},\\
		\mathfrak{C}_2&=\{\ub\in\RR^{(\NN)}\mid s(\ub)=0\}=\RR\Sym(\eb_1-\eb_2),\\
		\mathfrak{C}_3&=\{\ub\in\RR^{(\NN)}\mid s(\ub)\ge0\}
		=\RR_{\ge0}^{(\NN)}+\RR\Sym(\eb_1-\eb_2),\\
		\mathfrak{C}_4&=\{\ub\in\RR^{(\NN)}\mid s(\ub)\le0\}
		=\RR_{\le0}^{(\NN)}+\RR\Sym(\eb_1-\eb_2).
	\end{align*}
	In particular, every non-pointed $\Sym$-invariant cone in $\RR^{(\NN)}$ is $\Sym$-equivariantly finitely generated.
\end{theorem}

Before proceeding to the proof of this result, let us verify the equivalence of the two descriptions given for each of the cones $\mathfrak{C}_2,\mathfrak{C}_3,\mathfrak{C}_4$. Observe that for all $\ub,\vb\in \RR^{(\NN)}$, $\sigma\in\Sym$, and $\lambda, \mu\in \RR$, the following identities hold:
\[
s(\sigma(\ub))= s(\ub)
\quad\text{and}\quad
s(\lambda\ub+\mu\vb)=\lambda s(\ub)+\mu s(\vb).
\]
It follows that the second description of each cone is contained in the first. For the reverse inclusion, one only needs to check that
\[
\ub=s(\ub)\eb_1+\sum_{i\in\NN}u_i(\eb_i-\eb_1)
\]
for any $\ub=(u_i)_{i\in\NN}\in\RR^{(\NN)}$.

\Cref{t:non-pointed-cones} yields the following classification of   $\Sym$-invariant vector subspaces of $\RR^{(\NN)}$.

\begin{proposition}
	\label{p:sym-space}
	With notation as in \Cref{t:non-pointed-cones}, the only nonzero $\Sym$-invariant vector subspaces of $\RR^{(\NN)}$ are $\mathfrak{C}_1$ and $\mathfrak{C}_2$.
\end{proposition}

In fact, the proof of \Cref{t:non-pointed-cones} will rely on \Cref{p:sym-space}. Therefore, we first establish the latter. To this end, we require the following characterization of $\RR^{(\NN)}$ as a $\Sym$-invariant cone.

\begin{lemma}
	\label{l:big-cone}
	Let $C\subseteq\RR^{(\NN)}$ be a $\Sym$-invariant cone. Then the following statements are equivalent:
	\begin{enumerate}
		\item
		$C=\RR^{(\NN)}$;
		\item 
		There exist $\ub,\vb\in C$ such that $s(\ub)>0$
		and $s(\vb)<0$.
	\end{enumerate}
\end{lemma}

\begin{proof}
	For each $n\ge 1$, let $\epb_n=\eb_1+\cdots+\eb_n$. 
	We first prove the following claim.
	
	\begin{claim}
		\label{c:entire-cone}
		Let $\wb\in C$ with $s(\wb)\ne0$, and denote by $\omega$ the sign of $s(\wb)$. Then 
		\[
		\omega\epb_n \in C
		\quad\text{for all }\
		n\ge|\supp(\wb)|.\]
	\end{claim}
	
	\begin{proof}[Proof of \Cref{c:entire-cone}]
		By permuting the coordinates of $\wb$, if necessary, we may assume that $\wb\in\RR^m$, where $m=|\supp(\wb)|$. For any $n\ge m$, we compute
		\[
		\frac{1}{(n-1)!}\sum_{\sigma\in\Sym(n)}\sigma(\wb)
		=s(\wb)\sum_{i=1}^n\eb_i
		=s(\wb)\epb_n.
		\]
		Since the cone $C$ is $\Sym$-invariant, it follows that $s(\wb)\epb_n\in C$, and consequently, $\omega\epb_n\in C$ for all $n\ge m$.
	\end{proof}
	
	Let us now proceed with the proof of the lemma. Evidently, only the implication (ii)$\Rightarrow$(i) requires justification. Set $l=\max\{|\supp(\ub)|,|\supp(\vb)|\}$. Then it follows from  \Cref{c:entire-cone} that $\pm \epb_n\in C$ for all $n\ge l$. Hence, for all $n\ge l+1$, we have
	\[
	\eb_n=\epb_n+(-\epb_{n-1})\in C
	\quad\text{and}\quad
	-\eb_n=\epb_{n-1}+(-\epb_n)\in C.
	\]
   Since $C$ is $\Sym$-invariant, we deduce that $\pm \eb_i\in C$ for all $i\in\NN$. Therefore,  $C=\RR^{(\NN)}$.
\end{proof}

We also need the following simple property of $\Sym$-invariant vector spaces.

\begin{lemma}
	\label{l:subspace}
	Let $L\subseteq \RR^{(\NN)}$ be a $\Sym$-invariant vector space. If there exists $\ub=(u_i)_{i\in\NN}$ with $u_j\ne u_k$ for some $k\ne j$, then $\mathfrak{C}_2\subseteq L$.
\end{lemma}

\begin{proof}
	By permuting the coordinates of $\ub$, we may assume that $u_1\ne u_2$. Let $\sigma\in \Sym$ be the transposition that swaps $1$ and $2$. Then
	\[
	\sigma(\ub)=(u_2,u_1,u_3,\dots)
	\in L.
	\] 
	It follows that
	\[
	(u_1-u_2)(\eb_1-\eb_2)=\ub-\sigma(\ub)\in L,
	\]
	and consequently, $\eb_1-\eb_2\in L$. This yields $\mathfrak{C}_2=\RR\Sym(\eb_1-\eb_2)\subseteq L$, as claimed.
\end{proof}

We are ready for the proof of \Cref{p:sym-space}.

\begin{proof}[Proof of \Cref{p:sym-space}]
	Let $L\subsetneq\RR^{(\NN)}$ be any nonzero $\Sym$-invariant vector subspace. We aim to show that $L=\mathfrak{C}_2$. Take any $\ub=(u_i)_{i\in\NN}\in L\setminus\{0\}$. Then $u_i\ne0$ for some $i\in\NN$. Since $u_n=0$ for $n\gg0$, it follows from \Cref{l:subspace} that $\mathfrak{C}_2\subseteq L$. Conversely, for any $\vb\in L$, we have $-\vb\in L$ and $s(-\vb)=-s(\vb)$. Since $L\ne\RR^{(\NN)}$, \Cref{l:big-cone} implies that $s(\vb)=0$. Hence, 
	$L\subseteq \mathfrak{C}_2.$ 
	The proof is complete.
\end{proof}

Let us now prove \Cref{t:non-pointed-cones}.

\begin{proof}[Proof of \Cref{t:non-pointed-cones}]
	It suffices to prove the first assertion, as the following finite sets are  $\Sym$-equivariant generating sets for the cones $\mathfrak{C}_1,\mathfrak{C}_2,\mathfrak{C}_3,\mathfrak{C}_4$, respectively:
	\[
	\mathfrak{G}_1=\{\pm\eb_1\},\ \ \mathfrak{G}_2=\{\pm(\eb_1-\eb_2)\},
	\ \ \mathfrak{G}_3=\{\eb_1,\pm(\eb_1-\eb_2)\},
	\ \ \mathfrak{G}_4=\{-\eb_1,\pm(\eb_1-\eb_2)\}.
	\]
	Let $C\subsetneq\RR^{(\NN)}$ be any non-pointed $\Sym$-invariant cone. Then $\lin(C)$ is a nonzero proper $\Sym$-invariant vector subspace of $\RR^{(\NN)}$. So $\lin(C) =\mathfrak{C}_2$ by \Cref{p:sym-space}. It is thus enough to show that if $\lin(C) \subsetneq C$, then either $C=\mathfrak{C}_3$ or $C=\mathfrak{C}_4$. Indeed, since $\mathfrak{C}_2\subsetneq C$, there exists $\ub=(u_i)_{i\in\NN}\in C$ such that $s(\ub)\ne0$. Consider first the case $s(\ub)>0$. Then from \Cref{l:big-cone} it follows that $s(\vb)\ge0$ for all $\vb\in C$, i.e., $C\subseteq\mathfrak{C}_3$. On the other hand, we have
	\[
	s(\ub)\eb_1=\ub+\sum_{i\in\NN}u_i(\eb_1-\eb_i)\in C,
	\]
	since $\sum_{i\in\NN}u_i(\eb_1-\eb_i)\in \mathfrak{C}_2\subseteq C$.  This implies $\eb_1\in C$, and therefore,
	\[
	\mathfrak{C}_3=\RR_{\ge0}^{(\NN)}+\mathfrak{C}_2
	=\cone(\Sym(\eb_1))+\mathfrak{C}_2
	\subseteq C.
	\] 
	We have thus shown that $C=\mathfrak{C}_3$ if $s(\ub)>0$. When $s(\ub)<0$, an analogous argument yields $C=\mathfrak{C}_4$. This concludes the proof.
\end{proof}

Motivated by \Cref{t:non-pointed-cones}, it is natural to investigate the classification of non-pointed $\Sym(n)$-invariant cones in $\RR^n$ for $n\ge1$. In the case $n=1$, the only non-pointed cone is trivially the entire space $\RR$. For $n\ge2$, one might expect that the restrictions of the four global cones from \Cref{t:non-pointed-cones} would exhaust all possibilities. However, somewhat surprisingly, a fifth cone emerges. The next result may be viewed as a local counterpart of \Cref{t:non-pointed-cones}.

\begin{proposition}
	\label{p:non-pointed-local-cones}
	For $n\ge 2$, there are exactly five non-pointed $\Sym(n)$-invariant cones in $\RR^{n}$, namely,
	\begin{align*}
		\mathfrak{D}_1&=\RR^{n},\\
		\mathfrak{D}_2&=\{\ub\in\RR^{n}\mid s(\ub)=0\}=\RR\Sym(n)(\eb_1-\eb_2),\\
		\mathfrak{D}_3&=\{\ub\in\RR^{n}\mid s(\ub)\ge0\}
		=\RR_{\ge0}^{n}+\RR\Sym(n)(\eb_1-\eb_2),\\
		\mathfrak{D}_4&=\{\ub\in\RR^{n}\mid s(\ub)\le0\}
		=\RR_{\le0}^{n}+\RR\Sym(n)(\eb_1-\eb_2),\\
		\mathfrak{D}_5&=\{(a,\dots,a)\in\RR^{n}\mid a\in\RR\}
		=\RR\epb_n,
	\end{align*}
	where $\epb_n=\eb_1+\cdots+\eb_n.$
\end{proposition}

The proof of this proposition closely follows the argument used for  \Cref{t:non-pointed-cones}, with some minor modifications. Therefore, we will outline only the main steps, highlighting the necessary adjustments required for the local setting. Let us begin with a local counterpart of \Cref{l:big-cone}, which characterizes when a $\Sym(n)$-invariant cone equals the whole space.

\begin{lemma}
	\label{l:local-big-cone}
	Let $C_n\subseteq\RR^{n}$ be a $\Sym(n)$-invariant cone for some $n\ge 2$. Then the following conditions are equivalent:
	\begin{enumerate}
		\item
		$C_n=\RR^{n}$;
		\item 
		There exist $\ub,\vb\in C_n$ with $|\supp(\ub)|,|\supp(\vb)|<n$ such that $s(\ub)>0$
		and $s(\vb)<0$.
	\end{enumerate}
\end{lemma}

\begin{proof}
	It suffices to prove the implication (ii)$\Rightarrow$(i). Define $l=\max\{|\supp(\ub)|,|\supp(\vb)|\}$. Following the argument in the proof of \Cref{c:entire-cone}, we obtain $\pm\epb_k\in C_n$ for all $l\le k\le n$. Thus, in particular, $\pm\epb_n,\pm\epb_{n-1}\in C_n$ since $l<n$. A similar reasoning as in the proof of \Cref{l:big-cone} then implies that $\pm \eb_i\in C_n$ for all $i\in[n]$. This yields $C_n=\RR^{n}$, as desired.
\end{proof}

The next result identifies all $\Sym(n)$-invariant vector subspaces of $\RR^n$, showing how the local scenario differs from the global one established in \Cref{p:sym-space}.

\begin{lemma}
	\label{l:sym-space-local}
	Keep the notation from \Cref{p:non-pointed-local-cones}. Then for $n\ge2$, the only nonzero $\Sym(n)$-invariant vector subspaces of $\RR^{n}$ are $\mathfrak{D}_1$,  $\mathfrak{D}_2$, and $\mathfrak{D}_5$ .
\end{lemma}

\begin{proof}
	Let $L_n$ be a nonzero $\Sym(n)$-invariant vector subspaces of $\RR^{n}$. Suppose that $L_n\ne\mathfrak{D}_5$, i.e., there exists $(u_1,\dots,u_n)\in L_n$ with $u_i\ne u_j$ for some $i\ne j$. We claim that $L_n$ must be either $\mathfrak{D}_1$ or $\mathfrak{D}_2$. Indeed, following the argument in the proof of \Cref{l:subspace}, we obtain $\eb_1-\eb_2 \in L_n$, which yields $\mathfrak{D}_2\subseteq L_n$.  If $s(\vb)=0$ for all $\vb\in L_n$, then $L_n=\mathfrak{D}_2$. Otherwise, if $s(\vb)\ne0$ for some $\vb\in L_n$, we must show that  $L_n=\mathfrak{D}_1=\RR^n$. By \Cref{l:local-big-cone}, it suffices to find an element $\wb\in L_n$ satisfying  $|\supp(\wb)|<n$ and $s(\wb)>0$, as this would imply $-\wb\in L_n$ with $|\supp(-\wb)|<n$ and $s(-\wb)<0$. Since $\pm\vb\in L_n$ and $s(\vb)\ne0$, it follows from the proof of \Cref{c:entire-cone} that $\epb_n\in L_n$. Therefore,
	\[
	\wb\defas (\eb_1-\eb_2) + \epb_n=2\eb_1+\eb_3+\cdots+\eb_n\in L_n.
	\]
	As $|\supp(\wb)|=n-1$ and $s(\wb)=n$, we are done.
\end{proof} 

We are now prepared to prove \Cref{p:non-pointed-local-cones}.

\begin{proof}[Proof of \Cref{p:non-pointed-local-cones}]
	Suppose that $C_n\subsetneq\RR^{n}$ is a non-pointed $\Sym(n)$-invariant cone. Then $\lin(C_n)$ is either $\mathfrak{D}_2$ or $\mathfrak{D}_5$ by \Cref{l:sym-space-local}.
	
	\emph{Case 1}: $\lin(C_n)=\mathfrak{D}_2$. We claim that $C_n$ must be one of $\mathfrak{D}_2,\mathfrak{D}_3$, or $\mathfrak{D}_4$. As in the proof of \Cref{t:non-pointed-cones}, it suffices to show that $C_n$ is contained in either $\mathfrak{D}_3$ or $\mathfrak{D}_4$. If this were not the case, there would exist $\ub,\vb\in C_n$ such that $s(\ub)>0$ and $s(\vb)<0$. It then follows from the proof of \Cref{c:entire-cone} that $\pm\epb_n\in C_n$. Since $\pm(\eb_1-\eb_2)\in \mathfrak{D}_2\subseteq C_n$, we deduce that $\pm(\eb_1-\eb_2+\epb_n)\in C_n$. This together with \Cref{l:local-big-cone} implies that $C_n=\RR^{n}$, a contradiction.
	
	\emph{Case 2}: $\lin(C_n)=\mathfrak{D}_5$. We show that $C_n=\mathfrak{D}_5$. Suppose, for contradiction, that there exists $\ub=(u_1,\dots,u_n)\in C_n$ with $u_i\ne u_j$ for some $i\ne j$. By permuting the coordinates, we may assume that $u_1\le\cdots\le u_n$. Set $\vb=\ub-u_1\epb_n$ and $\wb=\ub-u_n\epb_n$. Then $\vb,\wb\in C_n$ since $\pm\epb_n\in \mathfrak{D}_5\subseteq C_n$. It is clear that $|\supp(\vb)|,|\supp(\wb)|<n$ and $s(\vb)>0,s(\wb)<0$. Applying \Cref{l:local-big-cone}, we obtain $C_n=\RR^{n}$. This contradiction completes the proof.
\end{proof}

\subsection{Non-positive symmetric normal monoids}

Building on the classifications of non-pointed symmetric cones established above, we now derive classifications for non-positive symmetric normal monoids. Let us begin with a characterization of positivity for monoids in terms of their associated cones. This slightly generalizes a classical result for affine monoids (see, e.g., \cite[Proposition 2.16]{BG}). Recall that a monoid $M\subseteq\ZZ^{(\NN)}$ is said to be positive if its group of units $\U(M)$ is trivial.

\begin{lemma}
	\label{l:positive-pointed}
	For any monoid $M\subseteq\ZZ^{(\NN)}$,  the following equality holds:
	$$\lin(\cone(M))=\RR\U(M).$$
	In particular, $M$ is positive if and only if $\cone(M)$ is pointed.
\end{lemma}

\begin{proof}
	It suffices to establish the first assertion. The inclusion $\RR\U(M)\subseteq\lin(\cone(M))$ is immediate. Conversely, take any $\ub\in \lin(\cone(M)).$ Since both $\ub$ and $-\ub$ lie in $\cone(M)$, they can be written as nonnegative linear combinations of finitely many elements of $M$. Thus, there exists an affine submonoid $M'\subseteq M$ such that $\pm\ub\in \cone(M')$, which implies $\ub\in \lin(\cone(M')).$ Because $\cone(M')$ is finitely generated, it lies in $\RR^n$ for some $n\ge1$. Furthermore, by the Minkowski--Weyl theorem, this rational cone can be expressed as the intersection of finitely many linear rational halfspaces in $\RR^{n}.$ Evidently, $\lin(\cone(M'))$ is the intersection of the hyperplanes defining these halfspaces, i.e., it is the solution space of a system of linear equations with integer coefficients. It follows that $\lin(\cone(M'))$ is spanned by $N\defas\lin(\cone(M'))\cap \ZZ^{n}$. For any $\vb\in N$, there exists $k\in\NN$ such that $\pm k\vb\in M'$, which implies $\vb\in \RR\U(M')$. Hence, 
	$
	N\subseteq \RR\U(M').
	$
	Therefore,
	$$\ub\in \lin(\cone(M'))=\RR N\subseteq \RR\U(M')\subseteq \RR\U(M).$$
	This shows that $\lin(\cone(M))\subseteq \RR\U(M)$, completing the proof.
\end{proof}

Combining the preceding lemma with \Cref{t:non-pointed-cones}, we obtain the following classification of all non-positive $\Sym$-invariant normal monoids. 

\begin{corollary}
	\label{c:non-positve-monoids}
	There are exactly four non-positive $\Sym$-invariant normal monoids in $\ZZ^{(\NN)}$:
	\begin{align*}
		\mathfrak{M}_1&=\ZZ^{(\NN)},\\
		\mathfrak{M}_2&=\{\ub\in\ZZ^{(\NN)}\mid s(\ub)=0\}=\ZZ\Sym(\eb_1-\eb_2),\\
		\mathfrak{M}_3&=\{\ub\in\ZZ^{(\NN)}\mid s(\ub)\ge0\}
		=\ZZ_{\ge0}^{(\NN)}+\ZZ\Sym(\eb_1-\eb_2),\\
		\mathfrak{M}_4&=\{\ub\in\ZZ^{(\NN)}\mid s(\ub)\le0\}
		=\ZZ_{\le0}^{(\NN)}+\ZZ\Sym(\eb_1-\eb_2).
	\end{align*}
	In particular, every non-positive $\Sym$-invariant normal monoid in $\ZZ^{(\NN)}$ is $\Sym$-equivariantly finitely generated.
\end{corollary}

\begin{proof}
	As in the proof of \Cref{t:non-pointed-cones}, it suffices to establish the first assertion.
	Using the notation from \Cref{t:non-pointed-cones}, we note that $\mathfrak{C}_i\cap \ZZ^{(\NN)}=\mathfrak{M}_i$ for all $i\in[4]$. Hence, $\mathfrak{M}_i$ is a non-positive $\Sym$-invariant normal monoid for all $i\in[4]$. Conversely, let $M$ be any non-positive $\Sym$-invariant normal monoid in $\ZZ^{(\NN)}$. Then the normality of $M$ implies that $M=\cone(M)\cap\ZZ^{(\NN)}$. According to \Cref{l:positive-pointed}, $\cone(M)$ is a non-pointed cone. Therefore, by \Cref{t:non-pointed-cones}, $\cone(M)=\mathfrak{C}_i$ for some $i\in[4]$. Since $\mathfrak{C}_i\cap \ZZ^{(\NN)}=\mathfrak{M}_i$, we conclude that $M=\mathfrak{M}_i$ for some $i\in[4]$, as required.
\end{proof}

As a discrete analogue of \Cref{p:non-pointed-local-cones}, the following result classifies all non-positive $\Sym(n)$-invariant normal monoids in $\ZZ^n$. Since its proof parallels that of \Cref{c:non-positve-monoids}, we omit the details.

\begin{corollary}
	\label{c:non-positive-local-monoids}
	For $n\ge 2$, there are exactly five non-positive $\Sym(n)$-invariant normal monoids in $\ZZ^{n}$:
	\begin{align*}
		\mathfrak{N}_1&=\ZZ^{n},\\
		\mathfrak{N}_2&=\{\ub\in\ZZ^{n}\mid s(\ub)=0\}=\ZZ\Sym(n)(\eb_1-\eb_2),\\
		\mathfrak{N}_3&=\{\ub\in\ZZ^{n}\mid s(\ub)\ge0\}
		=\ZZ_{\ge0}^{n}+\ZZ\Sym(n)(\eb_1-\eb_2),\\
		\mathfrak{N}_4&=\{\ub\in\ZZ^{n}\mid s(\ub)\le0\}
		=\ZZ_{\le0}^{n}+\ZZ\Sym(n)(\eb_1-\eb_2),\\
		\mathfrak{N}_5&=\{(a,\dots,a)\in\ZZ^{n}\mid a\in\ZZ\}
		=\ZZ\epb_n.
	\end{align*}
\end{corollary}

We conclude this section with the following problem, which arises naturally from the preceding results.

\begin{problem}
	\label{pb:non-normal-monoids}
	Extend \Cref{c:non-positve-monoids,c:non-positive-local-monoids} to monoids that are not necessarily normal. 
\end{problem}


\section{Equivariant Gordan's lemma}
\label{sec:equi-Gordan}

Having established the necessary groundwork in the preceding sections, we are now in a position to prove the following equivariant analogue of Gordan's lemma. This result, which generalizes \cite[Theorem 6.1]{LR21}, completely resolves \Cref{cj:equi-Gordan}.

\begin{theorem}
	\label{t:Gordan}
	Let $C\subseteq\RR^{(\NN)}$ be a $\Sym$-equivariantly finitely generated rational cone. Then $M=C\cap\ZZ^{(\NN)}$ is a $\Sym$-equivariantly finitely generated normal monoid.
\end{theorem}

\begin{proof}
	If $M$ is non-positive, then it is $\Sym$-equivariantly finitely generated by  \Cref{c:non-positve-monoids}. Assume now that $M$ is positive. Let $G\subseteq\ZZ^{r}$ for some $r\ge1$ be a finite $\Sym$-equivariant generating set for $C$. Consider the $\Sym$-invariant chain of cones  $\C=(C_{n})_{n\geq 1}$ defined by
	\[
	C_{n}=
	\begin{cases}
	\{\nub\}&\text{if }  n<r,\\
	\cone(\Sym(n)(G))&\text{if } n\ge r.
	\end{cases}
	\]
	It is evident that $\C$ stabilizes with direct limit $C$, and each $C_n$ is a finitely generated rational cone. Now consider the associated $\Sym$-invariant chain of monoids $\M=(M_{n})_{n\geq 1}$, where $M_n=C_n\cap\ZZ^{n}$ for all $n\ge1$. Clearly, $M$ is the direct limit of this chain. Since $M$ is positive, each $M_n$ is also a positive monoid. Furthermore, by the classical Gordan's lemma, each $M_n$ is an affine normal monoid. Because the chain $\C$ stabilizes, \Cref{t:stabilizing-cones} ensures that statement (iii) of that theorem holds for $C$ and $\C$. In turn, this implies that statement (iv) of \Cref{c:finte-generation-monoid} holds for $M$ and $\M$. Therefore, by \Cref{c:finte-generation-monoid}, $M$ is $\Sym$-equivariantly finitely generated. 
\end{proof}

Given a $\Sym$-equivariantly finitely generated rational cone $C\subseteq\RR^{(\NN)}$, it is natural to ask how to find a finite $\Sym$-equivariant generating set for the monoid $M=C\cap\ZZ^{(\NN)}$. If $M$ is non-positive, then \Cref{c:non-positve-monoids} provides an explicit description of $M$, making it straightforward to produce such a generating set. When $M$ is positive, a finite $\Sym$-equivariant generating set for $M$ is implicitly given in the proofs of  \Cref{t:stabilizing-cones,c:finte-generation-monoid}. We now make this set explicit in the following result.

\begin{proposition}
	\label{p:equi-gen-monoid}
	Let $C\subseteq\RR^{(\NN)}$ be a pointed $\Sym$-invariant rational cone that admits a finite $\Sym$-equivariant generating set $G\subseteq\ZZ^r$ for some $r\ge1$. Consider the chain $\C=(C_{n})_{n\geq 1}$, where $C_{n}=\{\nub\}$ if $n<r$ and
	\[
	C_{n}=\cone(\Sym(n)(G))\quad \text{if }\ n\ge r.
	\]
	Let $M_n=C_n\cap\ZZ^{n}$ for $n\ge1$, and denote the Hilbert basis of $M_n$ by $\Hc_{n}$. Set 
	$$q=\max\{3r^2,\|\Hc_{3r^2}\|\}.$$
	Then $\Hc_{q}$ is a $\Sym$-equivariant generating set for the monoid $M=C\cap\ZZ^{(\NN)}$. 
\end{proposition}

\begin{proof}
Since $G\subseteq M$, it is clear that $C=\cone(M)$. Thus, by  \Cref{l:positive-pointed}, the monoid $M$ is positive. Arguing as in the proof of 	\Cref{t:Gordan}, we deduce that $\M=(M_{n})_{n\geq 1}$ is a $\Sym$-invariant chain of positive affine normal monoids whose direct limit is $M$.
By construction, $G$ is a $\Sym(n)$-equivariant generating set for $C_n$ for all $n\ge r$. Hence, from the proof of the implication (ii)$\Rightarrow$(iii) in \Cref{t:stabilizing-cones}, it follows  that both conditions in statement (iii) of that theorem hold for $C$ and $\C$ when $n\ge p\defas3r^2$. Consequently, both conditions in statement (iv) of \Cref{c:finte-generation-monoid} are satisfied for $M$ and $\M$ when $n\ge p$. Now, invoking the proof of the implications (iv)$\Rightarrow$(iii)$\Rightarrow$(v) in \Cref{t:stabilizing-cones}, we conclude that for all $n\ge q=\max\{p,\|\Hc_{p}\|\}$, every element $\ub\in\Hc_n$ has support size at most $q$. In particular, there exists $\sigma\in\Sym(n)$ such that
\[
\sigma(\ub)\in M\cap\RR^q=M_q.
\]
Since $\ub$ is irreducible in $M_n$, it follows that $\sigma(\ub)$ is irreducible in $M_q$, i.e., $\sigma(\ub)\in\Hc_q$. Therefore,  $\Hc_n\subseteq\Sym(n)(\Hc_q)$ for all $n\ge q$, which implies that $\Hc_{q}$ is a $\Sym$-equivariant generating set for $M$.
\end{proof}

\enlargethispage{\baselineskip}

The following local counterpart of \Cref{p:equi-gen-monoid} follows directly from that proposition and its proof.

\begin{corollary}
	\label{c:stability-index-monoid}
	Let $\C=(C_{n})_{n\geq 1}$ be a stabilizing $\Sym$-invariant chain of pointed, finitely generated rational cones with stability index $r$. Let $\M=(M_{n})_{n\geq 1}$ be the associated chain of monoids, where  $M_n=C_n\cap\ZZ^{n}$ for $n\ge1$. Denote the Hilbert basis of each $M_n$ by $\Hc_{n}$. Then the chain $\M$ stabilizes with stability index at most $q\defas\max\{3r^2,\|\Hc_{3r^2}\|\}$.
\end{corollary}

As the next example demonstrates, the bound for the stabilization of equivariant Hilbert bases given in \Cref{p:equi-gen-monoid,c:stability-index-monoid}, while seemingly loose, is in fact sharp.

\begin{example}
	Consider the $\Sym$-invariant cone $C$ that is $\Sym$-equivariantly generated by $G=\{(-1,a)\}$, where $a\ge2$ is an integer. Let $\C=(C_{n})_{n\geq 1}$ and $\M=(M_{n})_{n\geq 1}$ denote the associated chains of cones and monoids, respectively, as defined in \Cref{p:equi-gen-monoid}. Then the chain $\C$ stabilizes with stability index $r=2$. For $n \ge 2$, it can be shown that the set
	\[
	\Hc_n=\{\eb_1\}\cup
	\Big\{(-1,b_1,\dots,b_{n-1})\mid b_i\in\ZZ_{\ge0},\ \sum_{i=1}^{n-1}b_i=a\Big\}
	\]
	is a $\Sym(n)$-equivariant Hilbert basis for $M_n$. In particular, $\ub\defas(-1,1^{(a)})\in\Hc_{a+1}$. Since $\ub$ is irreducible in $M_{a+1}$ and has support size $a+1$, it is easily seen that
	\[
	\ub\in M_{a+1}\setminus \mn(\Sym(a+1)(M_a)).
	\]
	Thus, $\ind(\M)\ge a+1$, where $\ind(\M)$ denotes the stability index of $\M$. On the other hand, since every element of $\Hc_n$ has support size at most $a+1$, we have $\Hc_n\subseteq\Sym(n)(\Hc_{a+1})$ for all $n\ge a+1$. Hence, $\ind(\M)= a+1$, and $\Hc_{a+1}$ is a $\Sym$-equivariant generating set for the monoid $M=C\cap\ZZ^{(\NN)}$. 
	Now, since $\|\Hc_{n}\|=a+1$ for all $n\ge2$, \Cref{c:stability-index-monoid} yields
	\[
	\ind(\M)\le q\defas\max\{3r^2,\|\Hc_{3r^2}\|\}
	 =\max\{12,a+1\}.
	\]
	Therefore, if $a\ge11$, then $\ind(\M)=a+1=q$, and $q$ is the minimal index for which $\Hc_{q}$ is a $\Sym$-equivariant generating set for $M$.
	
%
%
\end{example}

	
	\section{Further open problems}
	\label{sec:problems}
	
	We propose here two more open problems concerning dual monoids and dual lattices that are inspired by the equivariant analogues of the Minkowski--Weyl theorem established in \Cref{sec-Minkowski-Weyl}.
	
	In analogy with cones, the \emph{dual monoid} of a monoid $M\subseteq\ZZ^{(\NN)}$ is defined as
	\[
	M^*=\{\vb\in \ZZ^{\NN}\mid \langle \ub,\vb\rangle\ge 0\
	\text{ for all }\ \ub\in M\},
	\]
	and for a fixed $n\in\NN$, the dual of a monoid $M_n\subseteq\ZZ^n$ is similarly given by
	\[
		M_n^*=\{\vb\in \ZZ^{n}\mid \langle \ub,\vb\rangle\ge 0\
		\text{ for all }\ \ub\in M_n\}.
	\]
	
	In light of \Cref{t:Minkowski-Weyl,c:global-dual-cone}, the following problem naturally arises.
	
	\begin{problem}
		\label{pb:dual-monoids}
		Let $\M=(M_{n})_{n\geq 1}$ be a $\Sym$-invariant chain of monoids with direct limit $M=\bigcup_{n\geq 1}M_n.$ Describe the global dual monoid $M^*$ and the chain $\M^*=(M_{n}^*)_{n\geq 1}$ of local dual monoids.
	\end{problem}
	
	A solution to this problem may provide valuable insight into the local-global behavior of monoids (see \Cref{pb:local-global-monoids}), given the critical role played by \Cref{t:Minkowski-Weyl} in the proofs of \Cref{t:finte-generation,t:stabilizing-cones}.
	
	Recall that a \emph{lattice} $L\subseteq\RR^{(\NN)}$ is a free abelian subgroup of $\RR^{(\NN)}$. Its \emph{dual lattice} is defined as
	\[
	L^*=\{\vb\in \RR^{\NN}\mid \langle \ub,\vb\rangle\in\ZZ\
	\text{ for all }\ \ub\in L\}.
	\]
	Similarly, for a fixed $n\in\NN$, the dual of a lattice $L_n\subseteq\RR^n$ is given by
	\[
	L_n^*=\{\vb\in \RR^{n}\mid \langle \ub,\vb\rangle\in\ZZ\
	\text{ for all }\ \ub\in L_n\}.
	\]
	Motivated by finiteness questions concerning Markov bases in algebraic statistics, a general framework for analyzing symmetric lattices has been developed in \cite{LR24}, with a particular focus on their equivariant generating sets, equivariant Markov bases, equivariant Gr\"{o}bner bases, and equivariant Graver bases. Within this context, \Cref{t:Minkowski-Weyl,c:global-dual-cone} suggest the following problem.
	
	\begin{problem}
		\label{pb:dual-lattices}
		Let $\L=(L_{n})_{n\geq 1}$ be a $\Sym$-invariant chain of lattices with limit $L=\bigcup_{n\geq 1}L_n.$ Describe the global dual lattice $L^*$ and the chain  of local dual lattices $\L^*=(L_{n}^*)_{n\geq 1}$. In particular, determine whether $L^*$ and the $L_n^*$ admit finite equivariant generating sets, equivariant Markov bases, equivariant Gr\"{o}bner bases, and equivariant Graver bases, when they exist.
	\end{problem}
	

\section*{Acknowledgement} 

The author is deeply grateful to Tim R\"{o}mer for many inspiring  discussions over the years. Part of this work was done during a visit to the Vietnam Institute for Advanced Study in Mathematics (VIASM), whose hospitality and support are warmly acknowledged.

\end{document}